\documentclass[12pt]{amsart}
\usepackage{amscd,amssymb,latexsym}
\usepackage{fullpage}

\newcommand{\Mdef}[2]{\newcommand{#1}{\relax \ifmmode #2 \else $#2$\fi}}

\newcommand{\cok}{\mathrm{cok}}
\newcommand{\rank}{\mathrm{rank}}

\newcommand{\supp}{\mathrm{supp}}
\newcommand{\im}{\mathrm{im}}
\newcommand{\sm }{\wedge}

\newcommand{\tensor}{\otimes}

\newcommand{\Hom}{\mathrm{Hom}}

\newcommand{\Ext}{\mathrm{Ext}}
\Mdef{\bhom}{\mathbf{\hat{H}om}}
\newcommand{\dgHom}{\mathrm{dgHom}}
\Mdef{\Mod}{\mathrm{mod}}

\newcommand{\st}{\; | \;}

\newtheorem{thm}{Theorem}[section]
\newtheorem{lemma}[thm]{Lemma}
\newtheorem{prop}[thm]{Proposition}
\newtheorem{cor}[thm]{Corollary}

\theoremstyle{definition}

\newtheorem{defn}[thm]{Definition}

\newtheorem{example}[thm]{Example}

\newtheorem{remark}[thm]{Remark}

\newtheorem{convention}[thm]{Convention}

\newcommand{\ppf}{{\flushleft {\bf Proof:}}}
\newcommand{\qqed}{\qed \\[1ex]}
\renewenvironment{proof}[1][\hspace*{-.8ex}]{\noindent {\bf Proof #1:\;}}{\qqed}

\Mdef{\PH} {\Phi^H}
\Mdef{\PK} {\Phi^K}
\Mdef{\PL} {\Phi^L}
\Mdef{\PT} {\Phi^{\T}}

\Mdef{\ef}{E{\cF}_+}
\Mdef{\etf}{\tilde{E}{\cF}}
\Mdef{\eg}{E{G}_+}
\Mdef{\etg}{\tilde{E}{G}}
\Mdef{\tf}{\T / \! \!  / {\cF}_+}

\newcommand{\etp}{\tilde{E}\cP}

\newcommand{\piH}{\pi^H}

\newcommand{\piA}{\pi^{\cA}}

\Mdef{\infl}{\mathrm{inf}}
\Mdef{\defl}{\mathrm{def}}
\Mdef{\res}{\mathrm{res}}
\Mdef{\ind}{\mathrm{inf}}

\Mdef{\univ}{\mathcal{U}}

\Mdef{\Fp}{\mathbb{F}_p}
\Mdef{\Zpinfty}{\Z /p^{\infty}}
\Mdef{\Zpadic}{\Z_p^{\wedge}}

\newcommand{\bi}{\begin{itemize}}
\newcommand{\be}{\begin{enumerate}}
\newcommand{\bc}{\begin{center}}
\newcommand{\bd}{\begin{description}}
\newcommand{\ei}{\end{itemize}}
\newcommand{\ee}{\end{enumerate}}
\newcommand{\ec}{\end{center}}
\newcommand{\ed}{\end{description}}
\newcommand{\dichotomy}[2]{\left\{ \begin{array}{ll}#1\\#2 \end{array}\right.}
\newcommand{\trichotomy}[3]{\left\{ \begin{array}{ll}#1\\#2\\#3 \end{array}\right.}

\newcommand{\adjunction}[4]{
\diagram
#1:#2 \rrto<0.7ex> &&
#3  \llto<0.7ex> :#4 
\enddiagram}
\newcommand{\lra}{\longrightarrow}
\newcommand{\lla}{\longleftarrow}

\newcommand{\spec}{\mathrm{spec}}
\newcommand{\abgp}{\mathbf{AbGp}}
\newcommand{\rings}{\mathbf{Rings}}

\Mdef{\we}{\mathbf{we}}
\Mdef{\fib}{\mathbf{fib}}
\Mdef{\cof}{\mathbf{cof}}
\Mdef{\BI}{\mathcal{BI}}

\newcommand{\cofibre}{\mathrm{cofibre}}

\newcommand{\colim}{\mathop{  \mathop{\mathrm {lim}} \limits_\rightarrow} \nolimits}
\newcommand{\holim}{\mathop{ \mathop{\mathrm {holim}} \limits_\leftarrow} \nolimits}

\Mdef{\A}{\mathbb{A}}
\Mdef{\B}{\mathbb{B}}
\Mdef{\C}{\mathbb{C}}
\Mdef{\D}{\mathbb{D}}
\Mdef{\E}{\mathbb{E}}
\Mdef{\T}{\mathbb{T}}
\Mdef{\F}{\mathbb{F}}
\Mdef{\G}{\mathbb{G}}
\Mdef{\I}{\mathbb{I}}
\Mdef{\N}{\mathbb{N}}
\Mdef{\Q}{\mathbb{Q}}
\Mdef{\R}{\mathbb{R}}
\Mdef{\bbS}{\mathbb{S}}
\Mdef{\Z}{\mathbb{Z}}

\Mdef{\bA}{\mathbb{A}}
\Mdef{\bB}{\mathbb{B}}
\Mdef{\bC}{\mathbb{C}}
\Mdef{\bD}{\mathbb{D}}
\Mdef{\bE}{\mathbb{E}}
\Mdef{\bF}{\mathbb{F}}
\Mdef{\bG}{\mathbb{G}}
\Mdef{\bH}{\mathbb{H}}
\Mdef{\bI}{\mathbb{I}}
\Mdef{\bJ}{\mathbb{J}}
\Mdef{\bK}{\mathbb{K}}
\Mdef{\bL}{\mathbb{L}}
\Mdef{\bM}{\mathbb{M}}
\Mdef{\bN}{\mathbb{N}}
\Mdef{\bO}{\mathbb{O}}
\Mdef{\bP}{\mathbb{P}}
\Mdef{\bQ}{\mathbb{Q}}
\Mdef{\bR}{\mathbb{R}}
\Mdef{\bS}{\mathbb{S}}
\Mdef{\bT}{\mathbb{T}}
\Mdef{\bU}{\mathbb{U}}
\Mdef{\bV}{\mathbb{V}}
\Mdef{\bW}{\mathbb{W}}
\Mdef{\bX}{\mathbb{X}}
\Mdef{\bY}{\mathbb{Y}}
\Mdef{\bZ}{\mathbb{Z}}

\Mdef{\cA}{\mathcal{A}}
\Mdef{\cB}{\mathcal{B}}
\Mdef{\cC}{\mathcal{C}}
\Mdef{\mcD}{\mathcal{D}} 
\Mdef{\cE}{\mathcal{E}}
\Mdef{\cF}{\mathcal{F}}
\Mdef{\cG}{\mathcal{G}}
\Mdef{\mcH}{\mathcal{H}} 
\Mdef{\cI}{\mathcal{I}}
\Mdef{\cJ}{\mathcal{J}}
\Mdef{\cK}{\mathcal{K}}
\Mdef{\mcL}{\mathcal{L}}

\Mdef{\cM}{\mathcal{M}}
\Mdef{\cN}{\mathcal{N}}
\Mdef{\cO}{\mathcal{O}}
\Mdef{\cP}{\mathcal{P}}
\Mdef{\cQ}{\mathcal{Q}}
\Mdef{\mcR}{\mathcal{R}}
\Mdef{\cS}{\mathcal{S}}
\Mdef{\cT}{\mathcal{T}}
\Mdef{\cU}{\mathcal{U}}
\Mdef{\cV}{\mathcal{V}}
\Mdef{\cW}{\mathcal{W}}
\Mdef{\cX}{\mathcal{X}}
\Mdef{\cY}{\mathcal{Y}}
\Mdef{\cZ}{\mathcal{Z}}

\Mdef{\tA}{\tilde{A}}
\Mdef{\tB}{\tilde{B}}
\Mdef{\tC}{\tilde{C}}
\Mdef{\tE}{\tilde{E}}
\Mdef{\tH}{\tilde{H}}
\Mdef{\tK}{\tilde{K}}
\Mdef{\tL}{\tilde{L}}
\Mdef{\tM}{\tilde{M}}
\Mdef{\tN}{\tilde{N}}
\Mdef{\tP}{\tilde{P}}

\Mdef{\Ab}{\overline{A}}
\Mdef{\Bb}{\overline{B}}
\Mdef{\Cb}{\overline{C}}
\Mdef{\Db}{\overline{D}}
\Mdef{\Eb}{\overline{E}}
\Mdef{\Hb}{\overline{H}}
\Mdef{\Ib}{\overline{I}}
\Mdef{\Kb}{\overline{K}}
\Mdef{\Lb}{\overline{L}}
\Mdef{\Mb}{\overline{M}}
\Mdef{\Nb}{\overline{N}}
\Mdef{\Qb}{\overline{Q}}
\Mdef{\Tb}{\overline{T}}

\Mdef{\qb}{\overline{q}}
\Mdef{\rb}{\overline{r}}
\Mdef{\tb}{\overline{t}}
\Mdef{\ub}{\overline{u}}
\Mdef{\vb}{\overline{v}}
\Mdef{\hc}{\hat{c}}
\Mdef{\he}{\hat{e}}
\Mdef{\hf}{\hat{f}}
\Mdef{\hA}{\hat{A}}
\Mdef{\hH}{\hat{H}}
\Mdef{\hJ}{\hat{J}}
\Mdef{\hM}{\hat{M}}
\Mdef{\hP}{\hat{P}}
\Mdef{\hQ}{\hat{Q}}

\Mdef{\bolda}{\mathbf{a}}
\Mdef{\boldb}{\mathbf{b}}
\Mdef{\boldD}{\mathbf{D}}

\Mdef{\fm}{\frak{m}}

\Mdef{\eps}{\epsilon}

\input{xypic}

\newcommand{\Gspec}{G\mbox{-{\bf spectra}}}
\newcommand{\spectra}{\mbox{-{\bf spectra}}}
\newcommand{\TC}{\cT \cC}

\newcommand{\cOcF}{\cO_{\cF}}
\newcommand{\cOcFH}{\cO_{\cF /H}}
\newcommand{\cOcFtH}{\cO_{\cF /\tH}}
\newcommand{\cOcFK}{\cO_{\cF /K}}

\newcommand{\cOcFL}{\cO_{\cF /L}}

\newcommand{\tcOHmod}{\mbox{tors-}\cO_H\mbox{-mod}}
\newcommand{\cOHmod}{\cO_H\mbox{-mod}}

\newcommand{\ecOGmod}{\mbox{e-}\cO_G\mbox{-mod}}
\newcommand{\cOcFHmod}{\cO_{\cF /H}\mbox{-mod}}
\newcommand{\cOmod}{\cO \mbox{-mod}}

\newcommand{\ecOmod}{\mbox{e-}\cO \mbox{-mod}}

\newcommand{\qccOmod}{\mbox{qc-}\cO \mbox{-mod}}
\newcommand{\qcecOmod}{\mbox{qce-}\cO \mbox{-mod}}

\newcommand{\ecOGLmod}{\mbox{e-}\cO_{G/L}\mbox{-mod}}
\newcommand{\qcUcOmod}{\mbox{qc-$U$-$\cO$-mod}}
\newcommand{\qceUcOmod}{\mbox{qce-$U$-$\cO$-mod}}
\newcommand{\qcBcOmod}{\mbox{qc-$B$-$\cO$-mod}}
\newcommand{\qceBcOmod}{\mbox{qce-$B$-$\cO$-mod}}

\newcommand{\torsRmod}{\mbox{tors-$R$-mod}}

\newcommand{\modules}{\mbox{-modules}}

\newcommand{\cAh}{\hat{\cA}}

\newcommand{\notH}{[\not \supseteq H]}
\newcommand{\notK}{[\not \supseteq K]}
\newcommand{\connsub}{\mathrm{ConnSub}}

\newcommand{\ConnSub}{\mathrm{ConnSub}}

\newcommand{\noth}{\not \supseteq H}

\newcommand{\etnh}{\tilde{E}[\noth]}
\newcommand{\TCG}{\cT \cC G}

\newcommand{\picA}{\pi^{\cA}}

\newcommand{\pK}{\phi^K}
\newcommand{\pL}{\phi^L}

\renewcommand{\PT}{\Phi^T}

\newcommand{\cEi}{\cE^{-1}}

\newcommand{\icE}{\cI\cE}

\newcommand{\vv}{\vee}

\newcommand{\cl}[1]{\overline{ \{ #1\} } }

\newcommand{\HomA}{\mathrm{Hom}_{\cA}}
\newcommand{\comp}[1]{#1^{\wedge}_{\cF}}
\newcommand{\inflGHG}{\mathrm{inf}_{G/H}^G}
\newcommand{\ev}{\mathrm{ev}}

\newcommand{\CHB}{\check{H}_B}

\newcommand{\set}[1]{\{ #1\}}

\newcommand{\erf}{E\langle F \rangle}
\newcommand{\lr}[1]{\langle #1 \rangle}

\newcommand{\Mbar}{\overline{M}}
\newcommand{\Hz}{H_1}
\newcommand{\Kz}{K_1}

\newcommand{\efp}{E\cF_+}
\newcommand{\defp}{DE\cF_+}
\newcommand{\siftyK}{S^{\infty V(K)}}
\newcommand{\siftyH}{S^{\infty V(H)}}
\newcommand{\etnotK}{\tilde{E}\notK}

\renewcommand{\deg}{DEG_+}
\newcommand{\egh}{EG/H_+}
\newcommand{\decf}{DE\cF_+}
\newcommand{\decfh}{DE\cF /H_+}
\newcommand{\decfk}{DE\cF /K_+}
\newcommand{\piG}{\pi^G}
\renewcommand{\piH}{\pi^H}
\newcommand{\piK}{\pi^K}
\newcommand{\piGH}{\pi^{G/H}}
\newcommand{\piGK}{\pi^{G/K}}

\newcommand{\piAG}{\pi^{\cA (G)}}

\renewcommand{\PT}{\Phi^T}

\newcommand{\ertK}{E\langle \tK \rangle}
\newcommand{\erK}{E\langle K \rangle}

\newcommand{\ek}{E\langle K \rangle}
\newcommand{\etHH}{E\langle \tH /H \rangle}
\newcommand{\etKK}{E\langle \tK /K \rangle}
\newcommand{\sF}{\sigma_F^0}
\newcommand{\sH}{\sigma_H^0}
\newcommand{\sK}{\sigma_K^0}
\newcommand{\sL}{\sigma_L^0}
\newcommand{\efh}{E\cF /H_+}
\newcommand{\efk}{E\cF /K_+}
\newcommand{\inflGKG}{\mathrm{inf}_{G/K}^G}
\newcommand{\siftyc}{S^{\infty \bC}}

\newcommand{\freec}{S(\infty \bC)_+}

\newcommand{\id}{\mathrm{id}}

\newcommand{\spr}{\mathrm{spread}}

\begin{document}
\title[Rational torus-equivariant stable homotopy I]
{Rational torus-equivariant stable homotopy I:
calculating  groups of stable maps.}
\author{J.P.C.Greenlees}
\address{School of Mathematics and Statistics, Hicks Building, 
Sheffield S3 7RH. UK.}
\email{j.greenlees@sheffield.ac.uk}
\date{}

\begin{abstract}
We construct an abelian category $\cA (G)$ of sheaves over 
a category of closed subgroups of the $r$-torus $G$ and show it is
of finite injective dimension.  It can be used as a model 
for rational  $G$-spectra in the sense that there is a homology 
theory 
$$\piA_*: \Gspec \lra \cA (G)$$
on rational $G$-spectra with values in $\cA (G)$ and the associated
 Adams spectral sequence converges for all rational $G$-spectra and collapses
at a finite stage. 

\end{abstract}
\thanks{I am grateful to J.P.May and B.E.Shipley for numerous comments 
on early drafts of this paper. }
\maketitle

\tableofcontents

\part{Introduction.}

\section{Summary.}
\label{sec:Summary}
\subsection{The context.}
Non-equivariantly, rational stable homotopy theory is very simple: 
the homotopy category of rational spectra is equivalent to 
the category of graded rational vector spaces, and all cohomology 
theories are ordinary.
The author has conjectured that for each compact Lie group $G$, there
is an abelian category $\cA (G)$ so that the homotopy category of 
rational $G$-spectra is equivalent to the derived category of $\cA (G)$:
$$\mathrm{Ho}(\mbox{$G$-spectra}/\Q ) \simeq D(\cA (G)).$$
The conjecture describes various properties of $\cA (G)$, and in particular
asserts that its injective dimension is equal to the rank of $G$. 
Thus one can expect to  make complete calculations 
in rational equivariant stable homotopy theory, and one can classify cohomology 
theories. Indeed, one can construct a cohomology theory
by writing down an object in $\cA (G)$: this is how $SO(2)$-equivariant 
elliptic cohomology was constructed in \cite{ellT}, and work on curves of higher genus 
is underway.

The conjecture is elementary for finite groups where 
$\cA (G)=\prod_{(H)}\mbox{$\Q W_G(H)$-mod}$ \cite{Tate}, so that 
all cohomology theories are again ordinary. The conjecture has been 
proved for the rank 1 groups $G=SO(2), O(2), SO(3)$ in \cite{s1q, o2q, so3q}, 
where $\cA (G)$ is more complicated.
It is natural to go on to conjecture that the equivalence comes from a 
Quillen equivalence
$$\mbox{$G$-spectra}/\Q  \simeq dg \cA (G), $$
for suitable model structres, and Shipley has established this for $G=SO(2)$
\cite{Shipley}. The present paper is the first in a series \cite{tnq2,tnq3} 
intended to establish this Quillen equivalence, and hence the derived equivalence,
when $G$ is the $r$-dimensional torus for any $r \geq 0$.

\subsection{The results.}
The purpose of the present paper is to provide a means for
calculation in the homotopy category of rational $G$-spectra, where, 
for the rest of the paper, $G$ is an $r$-dimensional torus. In Part 2 
we construct an abelian category $\cA (G)$ (Definition \ref{defn:AG}) 
and in \ref{injdimisrank} show it is of finite injective dimension. The 
category $\cA (G)$ is a category of sheaves $M$ on the space of subgroups
of $G$, and the value $M(U)$ of a sheaf on a set $U$ of subgroups captures
the information about spaces with isotropy groups in $U$. In Part 3
we construct a homology theory
$$\piA_*: \Gspec \lra \cA (G)$$
with values in $\cA (G)$, and show that it is an effective
calculational tool in that there is an Adams spectral sequence.
The main theorem is as follows. 

\begin{thm}
There is a spectral sequence
$$\Ext_{\cA (G)}^{*,*}(\piA_*(X), \piA_*(Y)) \Rightarrow [X,Y]^G_*, $$
strongly convergent for all $G$-spectra $X$ and $Y$. 

The category $\cA (G)$ is of injective dimension $\leq 2r$, and so the 
spectral sequence is concentrated between rows $0$ and $2r$: it therefore 
collapses at the $E_{2r+1}$-page.
\end{thm}

\begin{remark}
We will show in \cite{tnq2} that in fact $\cA (G)$ is of injective
dimension precisely $r$, so that the spectral sequence collapses at the
$E_{r+1}$-page.
\end{remark}
The special case $r=1$ provided the basis for the results of \cite{s1q}.
In addition to being a  powerful tool, it is a perfectly practical 
one, since it is straightforward to make calculations in $\cA (G)$. The theorem essentially 
describes the category of rational $G$-spectra up to a finite
filtration. For many purposes this is quite sufficient, but 
the other papers in the series go further. In \cite{tnq3} Shipley and 
the author intend to combine the Adams spectral spectral sequence of the 
present paper with the  work of Schwede and Shipley \cite{SS} to show 
that the category of rational 
$G$-spectra is Quillen equivalent to $dg\cA (G)$. The paper
\cite{tnq2} provides the information about the algebraic structure
of the category $\cA (G)$ required in \cite{tnq3}. The principal 
results of the present paper were proved by 2001, but publication was 
delayed until the form of results required by \cite{tnq3} was clear.

\begin{convention}
Certain conventions are in force throughout the paper and the
series. The most important is that {\em everything is rational}: 
all  spectra  and homology theories are rationalized without
comment. The second is the standard one that `subgroup' means 
`closed subgroup'. We attempt to let inclusion of subgroups follow
the alphabet, so that when there are inclusions they are in 
the pattern $G \supseteq H \supseteq K \supseteq L$. The
other convention beyond the usual one that $\Hz$ denotes the identity 
component of $H$ is that $\tH$ denotes a subgroup with identity component
$H$ and $\hH$ denotes a subgroup in which $H$ is cotoral
(i.e., so that  $H\subseteq \hH$ and  $\hH /H$  is a torus).

Finally, cohomology is unreduced unless indicated to the contrary with a
tilde, so that $H^*(BG/K)=\tilde{H}^*(BG/K_+)$ is the unreduced cohomology ring.
\end{convention}

\subsection{Some standard constructions.}
We recall some standard constructions from equivariant homotopy theory.
As a general reference on stable equivariant homotopy theory, we use
\cite{lmsm}.

For any family $\mcH$ of subgroups (i.e., a  collection closed under passage
to conjugates and smaller subgroups), we may consider the 
$G$-space $E\mcH$, which is universal amongst $\mcH$-spaces; it is
 characterized by the fact that its $K$-fixed points
are either empty (if $K \not \in \mcH$) or contractible (if $K \in \mcH$). 
We then have a basic cofibre sequence
$$E\mcH_+ \lra S^0 \lra \tilde{E}\mcH, $$
and $\tilde{E}\mcH$ may also be constructed as the join $S^0 *E\mcH$.

The families playing a significant role for us are the family $\cF$
of all finite subgroups of $G$ and the family $\notK$
of subgroups not containing $K$. The role of $\notK$ is worth 
further comment. If $X$ is a $G$-space and we write  $\PK X$ for its
$K$-fixed points, it is clear that the inclusion $\PK X \lra X$ 
induces an equivalence
$$\etnotK \sm X \simeq \etnotK \sm \PK X .$$
In fact $\PK$ admits an extension to a functor 
$\PK : G\spectra \lra G/K\spectra$ (known as
the {\em geometric} fixed point functor), and the same
equivalence holds. 

This holds for any group $G$, provided $K$ is normal, but
for an abelian group there is a particularly convenient
construction of $\tilde{E}\notK$. Indeed, if 
we define
$$\siftyK =\colim_{V^K=0} S^V, $$
the special from of $G$ means
$$\tilde{E}\notK \simeq \siftyK .$$
(If $L$ does not contain $K$ then we can find 
a representation $V$ with $V^K=0$ and $V^L\neq 0$. Indeed, 
$K$ has non-trivial image in $G/L$, and hence 
$G/L$ has a character $\alpha$ with $\alpha^K=0$). 
This construction using spheres gives a connection with 
Euler classes and thence to commutative algebra. This is
a key ingredient in the results.  

Choosing an orientation of ordinary cohomology,  we
have Euler classes of vector bundles. Thus if $W$ is a representation of
$G/K$ then there is an associated Euler class $c_H(W) \in H^{|W|}(BG/K)$.
These enter the picture since $c_H(W)$ is the pullback of the 
Thom class $\tau_H(W)$ along the inclusion $S^0 \lra S^W$:
$$\begin{array}{ccc}
\tilde{H}^{|W|}(EG/K_+\sm_{G/K} S^W) & \lra& \tilde{H}^{|W|}(EG/K_+\sm_{G/K} S^0)\\
\tau_H(W)& \longmapsto &c_H(W).
\end{array}$$

\subsection{Outline of the argument.}
First we must construct the the category $\cA (G)$.
This is a category of sheaves on the space
of subgroups of $G$. In fact we consider the `natural' open sets
$$U(K)=\{ H \st H \supseteq K\}$$
of isotropy groups, where $K$ runs through the {\em connected}
subgroups of $G$, and
an object $M$ of $\cA (G)$ is specified by its values
$M(U(K))$  and the restriction maps $M(U(K)) \lra M(U(H))$
when $U(K) \supseteq U(H)$ (i.e., when $K  \subseteq H$). These 
are required to  satisfy certain conditions that we explain shortly, but
$M(U(K))$ contains information about isotropy groups $H$ in $U(K)$. 

It is rather easy to write down the functor $\piA_*(X)$.

\begin{defn}
\label{defn:piA}
For a $G$-spectrum $X$ we define $\piA_*(X)$ on $U$-open subsets by 
$$\piA_*(X)(U(K))=\pi^G_*(D\efp \sm \siftyK \sm X).$$
Here $\efp$ is the universal space for the family $\cF$ of finite subgroups with 
a disjoint basepoint added and $D\efp =F(\efp , S^0)$ is its functional dual
(the function spectrum of maps from $\efp$ to $S^0$).
Note that, since the $G$-space $\siftyK$ is defined by  
$$\siftyK =\colim_{V^K=0} S^V, $$
when $K \subseteq H$ there is a map 
$\siftyK \lra \siftyH$ inducing the restriction map 
$\piA_*(X)(U(K))\lra \piA_*(X)(U(H))$.\qqed
\end{defn}

The objects of $\cA (G)$ have the structure of  modules over the 
structure sheaf $\cO$ introduced formally in Subsection 
\ref{subsec:Structure}. The definition of the structure sheaf
is based on  the ring
$$\cOcF =\prod_{F\in \cF }H^*(BG/F_+), $$
where the product is over the family $\cF$ of finite subgroups of $G$. For 
this we use Euler classes: indeed if $V$ is a representation of $G$
we may defined $c(V) \in \cO_{\cF}$ by taking its components
$c(V)(F)=c_H(V^F) \in H^*(BG/F_+)$ to be classical ordinary homology 
Euler classes. 

The sheaf $\cO$ is defined by 
$$\cO (U(K))=\cEi_K \cOcF$$
where $\cE_K =\{ c(V) \st V^K=0\} \subseteq \cOcF$ is the multiplicative
set of Euler classes of $K$-essential representations.

To see that $\piA_*(X)$ is a module over
$\cO$, the key is to understand $S^0$.
 
\begin{thm}
\label{piASisO}
The image of $S^0$ in $\cA (G)$ is the structure sheaf:
$$\cO =\piA_*(S^0).$$
\end{thm}

We prove this in the course of  Sections \ref{sec:Basic}
to \ref{sec:piA}.

There are then two requirements
on $\cO$-modules to be objects of $\cA (G)$. Firstly they must 
be {\em quasi-coherent}, in that 
$$M(U(K))=\cEi_K M(U(1)),  $$
where $\cE_K$ is the set of Euler classes of $K$-essential representations
as before. The  definition of $\piA_*(X)$ shows that quasi-coherence for 
$\piA_*(X)$ is just a matter of understanding Euler classes, 
which we do in Section \ref{sec:Euler}.
The second condition involves the ring $\cOcF$ and  its analogue 
$$\cOcFK =\prod_{\tK \in \cF /K}H^*(BG/\tK_+)$$
for the quotient modulo a connected subgroup $K$, 
where $\cF /K$ is the family of subgroups $\tK$ of $G$
with identity component $K$.
The second condition is  that the object should be {\em extended}, in 
the sense that there is a specified isomorphism 
$$M(U(K))=\cEi_K \cOcF \otimes_{\cOcFK} \phi^K M$$
for some $\cOcFK$-module $\phi^KM$.
The extendedness of  $\piA_*(X)$  follows from a construction of the geometric
fixed point functor, and it turns out that 
$$\phi^K \piAG_*(X)= \pi^{G/K}_*(D\efp \sm \Phi^K(X)),$$
where $\PK$ is the geometric fixed point functor. This sketches the proof of 
the following result, proved formally in Section \ref{sec:piA}.

\begin{cor}
\label{piAinAG}
The functor $\piA_*$ takes values in the abelian category $\cA (G)$.
\end{cor}

This outlines the construction of the functor $\piA_*$. To 
construct the Adams spectral sequence we need to realize
an injective resolution of $\piA_*(X)$ in $\cA (G)$, and to 
prove the Adams spectral sequence works for maps into an 
injective. We must therefore first realize sufficiently many 
injectives. We show that there is a right adjoint $f_K$ to the evaluation 
at $K$ functor $\phi^K$ (described in detail in Subsection 
\ref{subsec:evalext}).
Thus for a suitable module $N$ over $\cOcFK$ we may form an object
$f_K(N)$ in $\cA (G)$. Taking $N=H_*(BG/\tK)$ for a subgroup 
$\tK$ with identity component $K$, viewed as an $\cOcFK$-module
via projection onto $H^*(BG/\tK)$, we obtain the object 
$I(\tK)=f_K(H_*(BG/\tK))$. This is injective since $H_*(BG/\tK)$
is injective over $H^*(BG/\tK)$.
It turns out (Lemma \ref{piAehisinj}) that a suspension of 
$I(\tK)$ is realized by the $G$-space  $\ertK$ defined in 
\ref{defn:erK}  in the sense that 
$$I(\tK )=\piA_*(\Sigma^{-c}\ertK )$$
where $\tK$ is of codimension $c$.
Next we  need to understand maps into injectives, showing that 
$$\piA_*:[X,I]^G_* \stackrel{\cong} \lra \HomA (\piA_*(X) , \piA_*(I))$$
for these sufficiently many injectives $I$. This constructs a spectral 
sequence with the correct $E_2$-term. Finally we must show convergence
by showing that $\piA_*(X)=0$ implies $X \simeq *$. This is an 
easy consequence of the geometric fixed point Whitehead theorem
\ref{univWhitehead}.

\section{Formal behaviour of equivariant homology theories.}
\label{sec:MotI}

There are two ways one may hope to encode data about the homology 
of fixed point sets. They are close enough to be confusing, so
it is worth making them explicit at the outset. We consider the case
that $G$ is a torus, and the reader may want to bear in mind the 
examples of stable homotopy and $K$-theory. 

Given  a $G$-equivariant homology theory $\tilde{F}^G_*(\cdot )$ and a $G$-space
$X$ we may consider the system of values
$$H \longmapsto \tilde{F}^{G/H}_*(\PH X),  $$
where $\PH X$ denotes the (geometric) $H$-fixed point set of $X$.
If $K \subseteq H$ there is an inclusion $\PH X \lra \PK X$
of $G/K$-spaces and hence a map $\tilde{F}^{G/K}_* (\PH X) \lra \tilde{F}^{G/K}_*(\PK X)$, 
but in general there will not be a map 
$\tilde{F}^{G/H}_*(\PH X ) \lra \tilde{F}^{G/K}_*(\PK X)$.
However in favourable circumstances there are maps of this sort, and 
we accordingly call a {\em contravariant} functor $M$ on subgroups an {\em 
inflation functor}. If we are just given the values $\tilde{F}^{G/H}_*(\PH X)$ and no 
structure maps between them we refer to an {\em inflation system}.
Because of the variance and the motivation we sometimes write
$M (G/H)$ to suggest dependence on the quotient group $G/H$.

On the other hand, we may always consider 
$$X \sm \tE\notH \simeq \PH X \sm \tE\notH$$
where $\notH$ is the family of subgroups not containing $H$. 
If $K \subseteq H$ then there is a natural map 
$\tE \notK \lra \tE \notH$, and hence a map 
$$\tilde{F}^G_*(\PK X \sm \tE \notK) \lra  \tilde{F}^G_*(\PH X \sm \tE \notH)  .$$
Now, 
$$\tE \notK =\colim_{V^K=0}S^V, $$
and, under orientability hypotheses,  $\tilde{F}^G_*(\PK X \sm \tE \notK) $ may be
expressed as a localization $\cEi_K \tilde{F}_*^G(X)$
of $\tilde{F}^G_*(X)$, where $\cE_K$ is some multiplicatively closed subset
of $F^G_*$, generated by ``Euler classes'' $e (V)$ with $V^K=0$.
We will call a {\em covariant} functor on subgroups of this form a 
{\em localization functor}. 

Two major differences should be emphasized. First,  inflation functors are
contravariant in the subgroup whilst localization functors are
covariant. Second,  a localization functor takes values which are modules over
$F^G_*$, whereas an inflation functor typically does not. 

When we are fortunate enough that $F$ gives an inflation functor and
also has a localization theorem it may happen that the two structures
are related in the sense that 
$$\tilde{F}^G_*(\PK X \sm \tE \notK)= \cEi_K F^G_* \tensor_{F^{G/K}_*}
\tilde{F}^{G/K}_*(\PK X) .$$
In other words, the favourable case is when we have the following
structure, which will be properly defined and axiomatized in 
later sections.
\be 
\item $R$, a ring-valued inflation functor (such as $K \longmapsto 
F^{G/K}_* $)
\item $M$ an inflation system, which is
 module valued functor over $R$, (such as $K \longmapsto 
\tilde{F}^{G/K}_*(\PK X) $)
\item $LM$ a localization functor, which is module valued
over $R(G/1)$, (such as $K \longmapsto \tilde{F}^G_*(\PK X \sm \tE \notK)$, and 
\item an isomorphism 
$$LM (K)=\cEi_K R(G/1) \tensor_{R(G/K)}M(G/K). $$
\ee
In this case we say that the  localization functor
$LM$ is {\em extended} with associated  inflation system $M$.
However, be warned that, even if $M$ is an inflation functor (i.e., it
has contravariant structure maps), this does not supply the structure
maps for $\cEi_K R(G/1) \tensor_{R(G/K)}M(G/K)$, so that $LM$ requires further
data.

\part{Categories of $U$-sheaves.}

The objects of the abelian category $\cA (G)$ are sheaves of modules over
a sheaf $\cO$ of rings. Accordingly we begin Section \ref{sec:Standard}
by describing the inflation 
functor on which the structure sheaf $\cO$ is based; we can then 
introduce Euler classes and proceed with the definition.
Once $\cA (G)$ is defined, we begin to control it in Section 
\ref{sec:Filtration}: first we import objects
from module categories, and then show that these suffice to build
all the objects and prove that $\cA (G)$ has finite injective dimension.

\section{The standard abelian category.}
\label{sec:Standard}
The present section leads up to the definition of the
standard model as a certain category of $U$-sheaves of
$\cO$-modules. Before we can express the definition 
we need to introduce the structure sheaf $\cO$, and
before we can do this (Subsection \ref{subsec:Structure})
 we need to describe its associated
inflation functor (Subsection \ref{subsec:FundInfl}) and
Euler classes (Subsection \ref{subsec:Euler}).

\subsection{The fundamental inflation functor.}
\label{subsec:FundInfl}

The entire structure we discuss is founded on the inflation functor
described in this section. We let $\ConnSub (G)$ denote the category 
of connected subgroups of $G$ and inclusions. 
An inflation functor is a contravariant  functor
$$M : \ConnSub (G) \lra \abgp  $$
We write $M_{G/H}$ for its value on $H$. 
The purpose of this section is to introduce a ring valued inflation functor
$$\cOcF : \ConnSub (G) \lra \rings, $$
whose value at $K$ is written $\cOcFK$.
Other notations can be convenient and have been used elsewhere, for example
$\cOcFK=\cO_{\cF (K)}=\cO (K)=R_{G/K}$, but we will stick to the above 
notation in this series. 

 For any connected subgroup $K$, we let
$$\cF /K =\{ \tK \st K \mbox{ of finite index in } \tK\}$$
denote the set of subgroups of $G$ with identity component $K$, which is in 
natural correspondence with the finite subgroups of $G/K$. Now take 
$$\cOcFK=\prod_{\tK \in \cF /K} H^*(BG/\tK)$$
where the product is over the set of subgroups with identity component
$K$.

To describe the inflation maps, suppose
 $K$ and $L$ are connected and $L \subseteq K$ and $L$ is of finite index in 
$\tL$. The inclusion defines a quotient map $q: G/L \lra G/K$ and hence
$$ q_*: \cF /L \lra \cF /K.$$

The  inflation map $\cOcFK \lra \cOcFL$ has $\tL$th component
$$\cOcFK=\prod_{\tK \in \cF /K} H^*(BG/\tK) 
\lra H^*(BG/q_*\tL) \lra  H^*(BG/\tL)$$
given by projection onto the term $H^*(BG/q_*\tL)$ followed by the inflation
map induced  by  the quotient $G/\tL  \lra G/q_*\tL$.


Now an {\em inflation system} of $\cOcF$-modules is given by specifying
 an $\cOcFK$-module $M_{G/K}$ for each subgroup $K$. No structure maps 
relating these modules are required.


\subsection{Euler classes.}
\label{subsec:Euler}

We are now in a position to describe 
the Euler classes which are used in the localization process. 
This will allow us to discuss localization functors, and hence 
 quasi-coherent and extended $U$-sheaves. 

The Euler class of an arbitrary representation is defined
in terms of those of simple representations using
 the product formula $e(V \oplus W )=e(V) e(W)$. 
Since $G$ is abelian,  it is enough to 
define Euler classes $e(\alpha)\in \cOcF$ for one dimensional 
representations $\alpha$. We take 
$$e(\alpha ) \in \cOcF=\prod_{F \in \cF } H^*(BG/F), $$
to be defined by 
$$e(\alpha)(F)=
\dichotomy{1 & \mbox{ if } \alpha^F=0}{c_1(\alpha)& \mbox{ if
$\alpha$ is trivial on $F$.}}$$

 This is not a homogeneous element. The best way to sanitize this is to 
introduce  an invertible sheaf associated to a representation (corresponding
to suspension), and make $e(\alpha)$ a section of that. Thus  $e(\alpha)$ 
should be thought of as a section of a line bundle vanishing
at a finite group $F$ if and only if $F$ acts trivially on 
$\alpha$. Since $H$ acts trivially on $\alpha $ 
if and only if  all finite subgroups of $H$ act trivially, we can think
of $e(\alpha)$ as defining the `$U$-closed' set of subgroups of 
$\ker (\alpha)$.

There are enough 
representations of $G$ in the sense that if $H$ is fixed, it is separated
from all subgroups (except those containing it) by an  Euler class:
  if $H \not \subseteq K$ there is a representation 
$\alpha$ trivial over $K$ and non-trivial over $H$. Accordingly, if $H$ is 
connected, the open set $U(H)$ of subgroups containing $H$ is defined by 
inverting the set 
$$\cE_H =\{ e(W ) \st W^H=0\}$$ 
of Euler classes of representations generated by characters not arising from $G/H$.
If $\tH$ has identity component $H$ we let $\cE_{\tH}=\cE_H$.

\begin{example}
For example if $G$ is the circle group and $z$ is the natural 
representation, $e(z)$ is supposed to define $\set{1}$. We 
think of $e(z)$ as the function (or rather global section) given on 
finite subgroups by
$$e(z)(F)=\dichotomy
{c & \mbox{ if } F=1}
{1 & \mbox{ if } F\neq 1.}$$
\end{example}

\begin{remark}
The correspondence with divisors can be very important 
(see \cite{ellT}).
By definition $e(\alpha) $ vanishes to the first order at 
finite subgroups of $\ker  (\alpha)$.  It is thus natural 
to view the line bundle of which $e(\alpha)$ is a 
generating section as corresponding to the `divisor' 
$\overline{\ker (\alpha)}$, and call it $\cO (-\ker (\alpha))$.  
\end{remark}

\subsection{The structure sheaf and the category $\protect\cA (G)$.}
\label{subsec:Structure}

We now turn to localization functors. We introduce terminology 
so that we can view them as giving sheaves of functions on the 
space of subgroups. 

For each closed connected subgroup $K$ of $G$ we consider the set $U(K)$
of subgroups containing $K$ (which 
can be identified with the set of subgroups of $G/K$). 
We view the collection 
$$\cU =\{ U(K) \st K \mbox{ a connected subgroup } \}$$
as the generating set for the $U$-topology on the set of subgroups of $G$. 
We carry the letter $U$ throughout the discussion  to distinguish it
from a second topology  introduced in \cite{tnq2}.

A $U$-sheaf $M$ is a contravariant functor $M: \cU \lra \abgp$. 
(The terminology is reasonable  since any cover of a set 
$U(K)$ by sets from $\cU$  must involve $U(K)$ itself, so the sheaf 
condition is automatically  satisfied).  Thus if $K $ and $L$ are connected
with $L \subseteq K$ then $U(L)\supseteq U(K)$ and there is a restriction map
$M(U(L)) \lra M(U(K))$. Note that this is covariant for the inclusion of 
subgroups and is therefore simply another  way of speaking of a localization 
functor. 

We may construct a $U$-sheaf from the ring $\cOcF$.

\begin{defn} 
(i) The structure $U$-sheaf $\cO$ is defined by 
$$\cO (U(K))=  \cEi_K \cOcF, $$
and the structure maps are the localizations.
Thus $\cO$ is a $U$-sheaf of rings, and its ring of global 
sections is $\cOcF$.\\
(ii) A sheaf of $\cO$-modules is a $U$-sheaf $M$ with the 
additional structure that  $M(U(H))$ is  
a module over $\cO (U(H))=\cEi_H\cOcF  $. The restriction maps 
are required to be module maps: if $L \subseteq K$,  the restriction map
$$M(U(L ) ) \lra M(U(K)), $$
for the inclusion  $U(L) \supseteq U(K)$ is required to be a map 
of $\cO (U(L))$-modules. 
\end{defn}

We shall be working almost
exclusively with sheaves $M$ of $\cO$-modules: the standard
model for rational $G$-spectra will be a category of dg sheaves 
of $\cO$-modules with additional structure. 

First we restrict attention to modules which are determined 
by their value on $U(1)$. This is analogous to forming a 
sheaf over $\spec (R)$ from an $R$-module $M$: its values over
the open set on which $x$ is invertible is $M[1/x]$. We also borrow
the well-established and unwieldy terminology from this situation. 

\begin{defn} A {\em quasi-coherent} $U$-sheaf (qc $U$-sheaf) of 
$\cO$-modules is one in which for each  connected subgroup $K$, the 
restriction map $M(U(1)) \lra M(U(K))$ is the map inverting the
multiplicatively closed set $\cE_{K}$. 
\end{defn}

\begin{remark} 
\label{rmk:qc}
(i) The structure sheaf $\cO$  is quasi-coherent.\\
(ii) For a quasicoherent sheaf, all values $M(U(H))$ are determined
by the value $M(U(1))$.\\
(iii) The quasi-coherence  condition has a major effect. For example, 
if $M$ is a  quasi-coherent  module only nonzero on  $U(1)$ then 
$M(U(1))$ is  necessarily a  torsion module, since it 
vanishes if we invert $\cE_K$ for any non-trivial $K$.
\end{remark}

The second restriction is to sheaves of $\cO$-modules 
which are extended from quotient groups in the following sense. 

\begin{defn}
A sheaf $M$ of $\cO$-modules is {\em extended} if we are given 
a tensor decomposition
$$M(U(K))=\cEi_K \cOcF \otimes_{\cOcFK} \pK M$$ 
where $\pK M$ is an $\cOcFK$-module, so that $\{ \phi^{\bullet}M\}$
is an inflation system of $\cOcF$-modules.
This splitting must be compatible with  restriction maps 
in that if $L \subseteq K$, the restriction is  obtained from a map 
$$\pL M \lra \cEi_{K/L} \cOcFL \otimes_{\cOcFK}\pK M$$
by extension of scalars. A morphism of extended modules is required
to arise from a map of inflation systems: 
if  $\theta : M \lra N$ is a morphism of extended modules, for each 
$K$ we have a diagram
$$\begin{array}{ccc}
M(U(K)) & \stackrel{\theta (U(K))} \lra & N(U(K))\\
=\downarrow &&\downarrow =\\
\cEi_K \cOcF \otimes \pK M &\stackrel{1 \otimes \pK \theta}
\lra & \cEi_K \cOcF \otimes \pK N .
\end{array}$$
 We write $\ecOmod$ for the category of extended $\cO$-modules.  
\end{defn}

\begin{remark} 
(i) The condition on restriction maps makes sense  since
$$\cEi_L \cOcF \otimes_{\cOcFL} \cEi_{K/L} \cOcFK \otimes_{\cOcFK}(\cdot)
=\cEi_K \cOcF \otimes_{\cOcFK}(\cdot) .$$
The point here is that representations $\alpha$ of $G/L$  
with $\alpha^{K/L}=0$ (whose Euler classes
lie in $\cE_{K/L}$)  map to representations of $G$ with $\alpha^K=0$
(whose Euler classes lie in $\cE_K$) under inflation.\\
(ii) The structure sheaf $\cO$ is extended
since 
$$\cEi_K\cOcF =\cEi_K \cOcF \otimes_{\cOcFK} \cOcFK .$$
(iii) The  splitting of $M(U(K))$ is specified by the {\em basing map} 
$$\pK M \lra M(U(K))$$
corresponding to the inclusion of the unit. \\
(iv) The reason for the notation is that $\pK M$ is analagous
to the value $E^{G/K}_*(\PK X)$ of a cohomology theory on geometric
fixed points (see also \ref{fixedpoints}).
\end{remark}

\begin{remark}
We may therefore think of a quasi-coherent extended $U$-sheaf $M$ of 
$\cO$-modules  as an $\cOcF$-module $M(U(1))$ together with additional 
structure. The additional structure specifies particular ``relative 
trivializations'' of $\cEi_{K}M(U(1))$:
$$\cEi_KM(U(1))=M(U(K))=\cEi_K \cOcF \otimes_{\cOcFK}\pK  M .$$
The whole structure is given by $M(U(1))$ together with 
 basing maps $\pK M  \lra \cEi_K M(U(1))$ giving the splittings.
\end{remark}

Finally, we may introduce the class of sheaves directly relevant to 
us. The  explicit identification of this category is, perhaps, the 
main achievement of this paper.

\begin{defn} 
\label{defn:AG}
The {\em standard abelian category} 
$$\cA =\cA (G) =\qcecOmod$$ 
is the category of all quasi-coherent
extended $U$-sheaves of $\cO$-modules (qce \cO-modules). 
We also use the notation
$$\cAh =\cAh (G) =\ecOmod$$
for the category of extended modules with no restriction on 
structure maps.  
\end{defn}

It is  useful to have an algebraic analogue of the fixed point functor.
This is defined on extended $\cO$-modules. 

\begin{lemma}
\label{fixedpoints}
There is a functor 
$$\PL : \ecOGmod \lra \ecOGLmod$$
defined by 
$$(\PL M) (U(K /L))=\cEi_{K/L} \cOcFL \otimes_{\cOcFK}
\pK M, $$
or equivalently  
$$\phi^{K/L} (\PL M)=\pK M.$$
This functor takes quasi-coherent modules to quasi-coherent modules.
\qqed
\end{lemma}

\begin{example}
The fixed point functor takes the $G$-structure sheaf $\cO$ to 
the $G/L$-structure sheaf: there is an equivalence 
$$\PL \cO_G = \cO_{G/L} $$
of sheaves of modules on the toral chain category for $G/L$. \qqed
\end{example}

\section{A filtration of the standard abelian category $\protect \cA (G)$.}
\label{sec:Filtration}

In this section we show that any object of the abelian category $\cA (G)$ 
can be built up from objects $f_H(N)$ arising from modules $N$ over the
rings $\cOcFH$ for various connected subgroups $H$. 
The object $f_H(N)$ is zero on $U(K)$ unless $K \subseteq H$, and
it is  constant where it is non-zero. 

The key to decomposing objects in this way 
is  the fact that all restriction maps $M(U(L)) \lra M(U(K))$ 
go in one direction: they even increase
the dimension of the subgroups.  The topological explanation of this 
phenomenon is just as in \cite{ratmack}. In equivariant stable homotopy
one expects to have to deal with restriction maps (decreasing the 
size of subgroups) and transfers (increasing the size of subgroups).
The restriction maps are built into the structure at an early stage.
Because we work over the rationals, the Burnside rings of finite groups
are products of copies of $\Q$, so that transfer maps 
for inclusions of finite index can be expressed entirely in terms of idempotents
from Burnside rings, so no maps increasing the size of subgroups are necessary.
The transfer maps for toral inclusions are zero, 
and the residual structure comes from the localization theorem.

This filtration is fundamental for calculation, and perhaps the 
first striking consequence is that the category $\cA (G)$ has 
finite injective  dimension. This is
the key to the power  of $\cA (G)$ in the study of 
$G$-equivariant cohomology theories.

Subsection \ref{subsec:evalext} introduces the method for constructing
objects of $\cA (G) $ from modules, Subsection \ref{subsec:constants}
shows how arbitrary objects can be constructed from these, and
Subsection \ref{subsec:HomlAlgebra} deduces  consequences  for homological
algebra. In \cite[8.1]{tnq2} we take this further to show the exact
injective dimension is the rank of $G$.

\subsection{Evaluation and extension.}
\label{subsec:evalext}

For a chosen connected subgroup $K$, evaluation gives a functor
$$\ev_K : \cOmod \lra \cO (U(K))\modules$$
defined by 
$$M \longmapsto M(U(K)). $$
This functor has a    right adjoint
$$c_K : \cO (U(K))\modules \lra \cOmod$$
given by taking the sheaf constant below $K$: 
$$c_K(N)(U(H))=
\dichotomy
{N & \mbox{ if } H \subseteq K}
{0 & \mbox{ if } H \not \subseteq K}$$
The unit of the adjunction  
$$\eta : M \lra c_K \ev_K M$$
is defined to be the restriction $\eta (U(L)):M(U(L)) \lra M(U(K))$ 
if $L \subseteq K$ and
is zero otherwise. The counit
$$\eps : ev_K c_K N \lra N$$ 
is the identity. Thus we have an adjunction 
$$\adjunction{{\ev_K}}{\cOmod}{\cO (U(K))\modules}{{c_K}} $$
with the  left adjoint on top. 

This adjunction obviously restricts to an adjunction between extended 
$\cO$-modules and extended $\cO (U(K))$-modules, and if we identify extended 
$\cO (U(K))$-modules
with modules for $\cOcFK$, this gives the adjunction 
$$\adjunction{{\pK}}{\ecOmod}{\cOcFK\modules}{{f_K}}. $$
Explicitly, $f_K(V)$ is constant below $K$ 
at $\cEi_K\cOcF \tensor_{\cOcFK}V$.
In other words, 
$$f_K(V)=c_K(\cEi_K\cOcF \tensor_{\cOcFK}V).$$

A little more care is necessary for quasi-coherent sheaves. Indeed, the 
$U$-sheaf $c_K(N)$ will not be quasi-coherent unless (i) $\cE_K$ is invertible 
on $N$ and (ii) $\cEi_{K'}N=0$ when $K'\not \subseteq K$. We call $\cOcF$-modules
satisifying (i) {\em $\cE_K$-invertible modules}, and those satisfying 
(ii) {\em $K$-torsion modules}. We call
sheaves with $M(U(K'))=0 $ when $K' \not \subseteq K$, {\em sheaves concentrated
below $K$}. Since quasi-coherent sheaves form a full subcategory,
this is the only obstacle,  and we have an adjunction
$$\adjunction
{{\ev_K}}
{\qccOmod\mbox{-below-$K$}}
{\mbox{$\cE_K$-inv-$K$-torsion-$\cOcF\modules$}}
{{c_K}}. $$

Finally, on qce $\cO$-modules we combine these to give the  adjunction
we actually need.

\begin{lemma}
For any connected subgroup $K$ there is an adjunction
$$\adjunction{{\pK}}{\qcecOmod\mbox{-below-$K$}}
{\mbox{torsion-$\cOcFK$}\modules}{{f_K}}. $$
Furthermore, for any torsion $\cOcFK$-module $V$ and an arbitrary 
extended module $M$, 
$$\Hom_{\cOcFK} (\pK M, V)=\Hom (M,f_K(V)).\qqed$$
\end{lemma}

\subsection{Finiteness of quasi-coherent sheaves.}

If $M$ is a sheaf of $\cO$-modules, then $M(U(K))$ is a module over the
ring $\cOcFK=\prod_{\tK} H^*(BG/\tK)$. This is not a particularly 
easy ring to work with, and in particular it is not Noetherian. 
However the quasi-coherence condition on modules means that the 
relevant module theory is much better behaved. This good behaviour
can mostly be traced back to the crucial lemmas proved in this section.

First we should formalize some terminology for $\cO$-modules, since
there are several notions that could loosely be called `support', 
but which should not be confused with the usual use of the word
in commutative algebra.

\begin{defn}
(i) If $M$ is a $U$-sheaf, 
$$\supp (M):= \{ K \in \connsub (G)\st M(U(K))\neq 0\}.$$ 

(ii) If $M$ is a $U$-sheaf, and $x \in M(U(L))$, 
$$\supp_L (x):= \{ K \in \connsub (G/L)\st x \mbox{ has non-zero image in } M(U(K))\}.$$ 

(iii) If $N$ is a module over $\cOcFK$ and $x \in N$, we say that 
the {\em spread} of $x$ is 
$$\spr(x):=\{ \tK \in \cF/K \st e_{\tK}x\neq 0\}.$$
\end{defn}

We have finiteness conditions on the support and on the spread of elements.
That on support is more elementary.

\begin{lemma}
\label{torsionsum}
Suppose $M$ is  a qce-sheaf,  $L$ is a connected subgroup of $G$
and  $x \in M(U(L))$. 
If  $\supp_L(x)$ consists of subgroups of dimension $\leq s$ then 
it contains only finitely many subgroups of dimension $s$.
\end{lemma}

\begin{proof}
All subgroups considered in this proof are assumed connected and to contain $L$.
The hypothesis states that for each subgroup $H$ of dimension more than $s$
there is a representation $V(H)$ with $V(H)^H=0$, and $e (V(H))x=0$. 
It follows that $x$ maps to zero in $M(U(K))$ whenever $V(H)^K=0$. 
Thus if $V(H)=\alpha_1(H) \oplus \cdots \oplus
\alpha_{n(H)}(H)$, we see that $K$ can only be in the support of $x$ if
$K \subseteq \ker (\alpha_i(H))$ for some $i$. 

For $r=s$ there is nothing to prove, so we suppose $s<r$.
We now argue by induction on the codimension $c$, that if $K \in \supp_L(x)$
there are only finitely many subgroups $H$ of codimension $c$ that contain 
$K$. The result for $c=r-s$ is the statement of the lemma.
For $c=0$ we see that $K$ must be contained in $\ker (\alpha_i(G))$ for some
$i$. Now suppose the result is proved in codimension $c<r-s$, and pick one 
of the finitely many subgroups $H$ of codimension $c$. Subgroups in the
 support inside $H$ must lie in one of the subgroups $\ker (\alpha_i(H))$, 
 and hence in one of the finitely many codimension $c+1$ 
subgroups $H \cap \ker (\alpha_i(H))$. This completes the inductive step.
\end{proof}

The proof of the result on spread uses  a result from Section \ref{subsec:HomlAlgebra}, 
but the result is stated here for ease of reference.
\begin{lemma}
\label{finitespread}
Suppose $M$ is  a qce-sheaf, $L$ is a connected subgroup of $G$
of dimension $s$ and  $x \in \phi^LM$.
If $1 \tensor x \in \ker (M(U(L))\lra M(U(K)))$ for all connected subgroups $K$ containing
$H$ of  dimension $s+1$  then $x$ has finite spread.
\end{lemma}

\begin{proof}
Since $\cEi_L\cOcF$ is flat over $\cOcFL$ by \ref{faithfulflatness}, it suffices to consider
the special case $L=1$.
The hypothesis states that for each $1$-dimensional 
connected subgroup $K$, there is a representation $V(K)$
with $V(K)^K=0$, and $e (V(K))x=0$. 

Now for each finite subgroup $F$ the Euler class
$e(\alpha)(F)$ is 1 if $\alpha^F=0$ and is $c(\alpha)\in H^2(BG/F)$ if 
$\alpha$ is a representation of $G/F$. Thus
 if $V(K)=\alpha_1(K)\oplus \cdots \oplus \alpha_{n(K)}(K)$, 
the $F$-coordinate of $e(V(K))$ is 1 unless $F \subseteq
\ker (\alpha_i)$ for some $i$. Hence, writing $\cF M$ for the set of finite
subgroups of $M$,  
$$\spr (x) \subseteq \cF \ker (\alpha_1(K)) \cup 
 \cF \ker (\alpha_2(K)) \cup \cdots \cup
 \cF \ker (\alpha_n(K)). $$
More generally, we argue by induction 
that for $i=1,2, 3, \ldots , r$
$$\spr (x)\subseteq \cF H^{i}_1 \cup \cF H^{i}_2\cup\cdots \cup \cF H^{i}_{m(i)}$$
for certain subgroups $H^i_j$ of codimension at least $i$. We have already dealt 
with the case $i=1$, and the case $i=r$ shows there are a finite number of subgroups
in $\spr (x)$. For the inductive step we suppose $i <r$. If all the subgroups
already have codimension at least $i+1$ we can take $H^{i+1}_j=H^i_j$. Otherwise, 
for each $j$ with $H^i_j$ infinite we can find $K_j \subseteq H^i_j$. The inductive
hypothesis together with the spread condition for the $K_j$ shows that we may take
the $H^{i+1}_k$ to be intersections of $H^i_j $ and the $\ker (\alpha_s(K_j))$;
since $K_j \not \subseteq \ker (\alpha_s(K_j))$, these subgroups are of 
codimension $\geq i+1$. This completes the inductive step. 
\end{proof}

\subsection{$U$-sheaves are constructed from constant ones.}
\label{subsec:constants}

The category of qce $U$-sheaves is an abelian category, and we 
will need to do homological algebra in it. 
We use the modules constant below some point to  import 
 convenient objects into the  category 
of $\cO$-modules from categories of modules over suitable rings. 
These objects suffice to build everything, and the method
for proof below gives a practical method of calculation. 



\begin{thm}
\label{Usheavesconstructible}
The  qce $U$-sheaves constant below connected subgroups 
(i.e., the sheaves of the form $f_K(V)$ for some connected subgroup 
$K$ and some torsion $\cOcFK$-module $V$) generate the category 
of all qce $U$-sheaves using short exact sequences and sums.
\end{thm}

\begin{proof}
We say that a $U$-sheaf is supported on a set of subgroups $\cK$ if 
$M(U(K'))=0$ when $K' \not \in \cK$. We argue by finite induction on $s$ that 
qce sheaves supported on subgroups of dimension $\leq s$ are generated
by $U$-sheaves constant below some point. The induction begins since 
the statement is obvious with $s=-1$, and the theorem is the case $s=r$. 

Suppose then that qce $U$-sheaves supported on subgroups of dimension 
$\leq s-1$ are generated by $U$-sheaves constant below some point, and
that $M $ is a qce $U$-sheaf supported on subgroups of dimension $\leq s$. 
For each connected subgroup $L$ of dimension $s$ we  note as in Remark
\ref{rmk:qc}(iii) that $M(U(L))$ 
is torsion and lift the identity map  $M(U(L)) $ to a map 
$M \lra f_L(M(U(L)))$. Now   combine these to a map 
$$M \lra \prod_{\dim (L)=s} f_L(M(U(L))) . $$
The product is the termwise product of vector spaces, and therefore
not a qce sheaf. 

The sheaf $M$ is supported in dimension $\leq s$ so the image of the 
map satisfies the hypotheses of \ref{torsionsum}, and the map into the product 
actually maps into the sum.  We thus obtain 
$$g:M \lra \bigoplus_{\dim (L)=s} f_L(M(U(L))) .$$
The first point is that the sum is a qce $U$-sheaf since
localization and tensor products commute with direct sum.

Next the  map $g$ is an isomorphism at $U(H)$ whenever $\dim (H) \geq s$.
The kernel and cokernel are then supported on subgroups of dimension $\leq s-1$ and
hence constructed from constant sheaves by induction. 
\end{proof}

\begin{remark}
If we view a qce-sheaf $M$ as a refinement of the module $M(U(1))$, the
construction described above is a refinement of a generalized Cousin
complex for $M(U(1))$. For example the first map 
$$\eta: M \lra f_G(M(U(G))$$
evaluated at $U(1)$ is 
$$\eta (U(1)): M(U(1)) \lra \cEi_GM(U(1)), $$
and the next step 
$$\cok (\eta) \lra \bigoplus_M f_M(\phi^M\cok (\eta)))$$
evaluated at $U(1)$ is 
$$\cok (\eta (U(1))) \lra \bigoplus_M \cEi_M \cok(\eta(U(1))).$$
At each stage the kernel is already understood by induction.
\end{remark}

\section{Homological algebra of categories of sheaves.}
\label{subsec:HomlAlgebra}

The fact that the category of qce-sheaves has finite injective 
dimension is fundamental for the convergence of our spectral sequence.
The fact that it is small makes the method of calculation practical. 
The idea is simple:  since the $f$ constructions 
are right adjoints to  evaluation we obtain a good supply of injectives, 
and since the cohomology rings $H^*(BG/K)$ are polynomial rings on 
at most $r$ generators the injective dimension is $r$. Implementing
this idea is a little more complicated because we need to work with 
non-Noetherian rings like $\cOcF$. 


\subsection{Sums of injectives.}

One of the convenient things about Noetherian rings is that 
arbitrary sums of injectives are injective. This is not true 
for the category of arbitrary modules over $\cOcF$. However 
the qce condition is sufficient to obtain the required property
for the categories relevant to us.

\begin{lemma}
\label{sumprod}
If $M$ is a qce-sheaf  and $N_L$ is an 
$\cOcFL$-module for each connected subgroup $L$ of dimension $s$
then 
\begin{multline*}
\Hom (M, \bigoplus_Lf_L(N_L))\stackrel{=}\lra \Hom (M, \prod_Lf_L(N_L))\\
=\prod_L\Hom (M, f_L(N_L))=\prod_L\Hom (M (U(L)), N_L).
\end{multline*}
The corresponding result holds if $L$ runs through all subgroups, and 
not just the connected ones. 
\end{lemma}

\begin{proof}
First consider the statement with $L$ running through connected subgroups.
Any map 
$$f:M \lra \bigoplus_L f_L(N_L).$$
is determined by its behaviour on the sets $U(L)$ because the map 
$$\bigoplus_L f_L(N_L)\lra \prod_L f_L(N_L)$$
of $U$-sheaves is a termwise monomorphism. 
 Now replace $M$ by its image, to obtain a sheaf supported in dimension 
$\leq s$. Applying \ref{torsionsum}, any map into the product factors through
the sum.

The same proof gives the conclusion for all subgroups, now quoting 
\ref{finitespread}. 
\end{proof}

This lets us identify a sufficient supply of injective modules.

\begin{cor}
\label{sumsofinjectives}
If $I_{\tL}$ is a torsion injective $H^*(BG/\tL)$-module for each subgroup
 $\tL$ with identity component $L$, 
then $\bigoplus_{\tL} f_L(I_{\tL})$
is an injective qce-sheaf. \qqed
\end{cor}

\subsection{The injective dimension of $\cA (G)$.}

In fact the injective dimension of $\cA (G)$ is equal to the rank 
$r$ of $G$. This will be proved in \cite[8.1]{tnq2}. For the construction 
and convergence of the Adams spectral sequence it is enough to show the 
injective dimension is finite. In fact, the proof below shows that 
the injective dimension is $\leq 2r$. One expects the method to
establish the exact  bound, but this would involve delicate arguments
to justify the behaviour of the category of modules occurring as $M(U(1))$ 
for qce-sheaves $M$.

\begin{thm}
\label{injdimisrank}
The category of qce $\cO$-modules has injective dimension $\leq 2r$.
\end{thm}

\begin{proof} We prove by induction on  $s$ that any qce-sheaf 
$M$ supported on subgroups of dimension $\leq s$ is of injective dimension 
at most $r+s$ ($\id (M) \leq r+s)$. 

If $M$ is a qce-sheaf supported in dimension 0, then $M=f_1(N)$ and
by \ref{sumprod} we have a decomposition 
$N =\bigoplus_F N_F$ with $N_F$ a torsion $H^*(BG/F)$-module. Since
each $N_F$ has an injective resolution of length $\leq r$ consisting
of torsion modules; indeed, any torsion module embeds in a sum of copies
of the injective hull of the residue field, and the cokernel is again 
a torsion module. We can add these injective modules and apply $f_1$ to obtain 
an injective resolution of $M$.

Now suppose $s \geq 1$, and that the theorem has been proved for sheaves 
supported in dimension $s-1$.
If $M$ is supported in dimension $s$, we may consider the connected subgroups
$L$ of dimension $s$ and the map
$$j: M \lra E=\bigoplus_{\dim L =s}f_{L}(\phi^LM).$$
of \ref{Usheavesconstructible}. 
By definition this map is an isomorphism 
at each subgroup $L$ of dimension $s$,  so that its kernel and cokernel are 
supported in dimension $\leq s-1$. Thus we have  two short exact sequences
$$0 \lra M' \lra M \lra I \lra 0$$
and
$$0 \lra I \lra E \lra M'' \lra 0. $$
Since $M'$ and $M''$ are supported in dimension $\leq s-1$ by construction, 
they have injective dimension $\leq r+s-1$ by induction. It therefore 
follows that 
$\id (M) \leq r+s $   if  $\id (I) \leq r+s$, so 
we may concentrate on the second exact sequence.

Next, $\phi^L M =\bigoplus_{\tL} M_{\tL}$, and by \ref{sumsofinjectives}, we see 
that the injective dimension of $E$ is the maximum of those of the $M_{\tL}$,
 and hence  $\leq r$ since this is true of $H^*(BG/\tL)$-modules. It follows
from the second exact sequence that $\id (I) \leq r+s$.
\end{proof}

\subsection{Realizing enough injectives.}
When we come to connections with $G$-spectra we need to know 
we can realize enough injectives, and accordingly it is good
to have  a small list of injectives.

\begin{lemma}
\label{enoughinjectives}
There are enough injective qce-sheaves of $\cO$-modules
which are sums of those of the form 
$$I(\tL )=f_{L}(H_*(BG/\tL))$$
 where $\tL$ is a subgroup with identity component $L$.
\end{lemma}

\begin{proof}
By Theorem \ref{Usheavesconstructible} and \ref{sumsofinjectives}, 
it suffices to observe
that qce-sheaves $f_L(N_{\tL})$ with $N_{\tL}$ a torsion $H^*(BG/\tL)$-module
have resolutions using modules of the form $f_L(I(\tL))$.
The module $H_*(BG/\tL)$ is the $\Q$-dual  of $H^*(BG/\tL)$ and  
therefore injective over it; indeed, it is the injective hull of the
residue field.  Any torsion $H^*(BG/\tL)$-module 
may be embedded in a sum of copies of this, and the quotient is again 
a torsion module. 
\end{proof}

\subsection{Hausdorff modules and flatness.}

It is essential in several places to know that the process of extension 
loses no information, which corresponds
to the idea that we can recover all $G/K$-information from the inflated
$G$-equivariant information by passage to geometric $K$-fixed points. 

First note that the rings $\cOcF$ and $\cOcFK$ are  products of polynomial 
rings. Quite generally, if $R=\prod_iR_i$ is a product of commutative
rings, and $M$ is an $R$-module, we write $M_i$ for the summand corresponding
to $R_i$. We say that an $R$-module $M$ is {\em Hausdorff} if the map
$M \lra \prod_i M_i$ is a monomorphism. To see this is a restriction, 
note that $(\prod_i M_i)/(\bigoplus_i M_i)$
is not Hausdorff if infinitely many of the $M_i$ are non-zero.

\begin{lemma}
\label{ffHausdorff}
The inflation map $\cOcFK \lra \cOcF$ and the localization 
$\cOcFK \lra \cEi_K \cOcF$ are faithfully flat on the category of 
Hausdorff modules. 
\end{lemma}

\begin{proof}
First, it is clear that individual inflation maps $H^*(BG/\tK)\lra H^*(BG/F)$
are free and hence faithfully flat. It follows 
that $H^*(BG/\tK) \lra \prod_{FK=\tK}H^*(BG/F)$ is faithfully flat. The map
$\cOcFK \lra \cOcF$ is a product of such maps. 

Quite generally if we have modules
$M_i$ faithfully flat over rings $R_i$, then $\prod_i M_i$ is faithfully flat
over $\prod_iR_i$ in the category of Hausdorff modules. For flatness, we note 
that the product is exact and the map 
$$(\prod_i M_i) \tensor N \lra \prod_i(M_i\tensor_{\prod_iR_i}N)$$
is a monomorphism. For faithful flatness, we note that $\prod_iM_i \tensor_{\prod_iR_i}N$ 
has $M_i \tensor_{R_i}N $ as a retract, so by faithful flatness, if the
former were zero then 
$e_iN=0$, and so by the Hausdorff property, $N$ would be zero.

Flatness of the localization follows since localization can be effected 
by a filtered colimit.  Faithfulness uses the special nature of the 
set of elements inverted. Indeed, the colimit diagram consists of 
multiplication by $e(W)$ where $W^K=0$. We claim that if $\alpha $ 
is a simple representation with $\alpha^K=0$ then  the map
$e(\alpha ): \cOcF \lra \cOcF$ is a split monomorphism of $\cOcFK$-modules.
Indeed, the map splits as a product over finite subgroups $F$, and it
suffices to prove the result for each factor. If $\alpha^F=0$ the map 
is the identity, which is certainly split mono. If $\alpha^F \neq 0$ the
map is 
$$c_1(\alpha): H^*(BG/F) \lra H^*(BG/F)$$ 
viewed as a map of modules over $H^*(BG/FK)$. Indeed, we may choose
a splitting $H^*(BG/F)=H^*(BG/FK) \tensor H^*(FK)$, and we may choose
polynomial generators for $H^*(FK)$, including $c_1(\alpha)$ amongst them.
Using the monomials as an $H^*(BG/FK)$-basis of $H^*(BG/F)$ it is clear
that the map is a split monomorphism.
\end{proof}

\begin{lemma}
If $X$ is a qce module then the modules $\phi^KX$ are all Hausdorff.
\end{lemma}

\begin{proof}
As in \cite[5.10.1]{s1q} we see that a submodule of a Hausdorff module 
is Hausdorff, and that the class of Hausdorff modules is closed under
taking sums, products and extensions. 

We argue by induction on the rank of $G$. The case when $G=1$ is
trivial. Now suppose $G$ is of rank $r \geq 1$ and that the result 
has been proved for smaller rank. In particular if $K \neq 1$, 
it follows from  the result for $G/K$ that $\phi^KX$ is Hausdorff.

Write $M=\phi^1X=X(U(1))$.  The argument that $M$ is Hausdorff is by induction on the 
dimension of support of $X$. If $X$ is supported in dimension 0 then 
$M=\bigoplus_F M_F$ by \ref{torsionsum}, so we may suppose $X$ is supported
in dimension $d \geq 1$ and the result has been proved for modules 
supported in smaller dimension. Now consider the exact sequence
$$0 \lra X' \lra X \lra \bigoplus_{\dim (K)=d}f_K(\phi^KX) \lra X'' \lra 0 $$
of \ref{injdimisrank}, 
and let $I=\im (X \lra \bigoplus_{\dim (K)=d}f_K(\phi^KX))$. By induction
$X'(U(1))$ is Hausdorff, so the exact sequence 
$$0 \lra X' \lra X \lra I \lra 0$$
shows it suffices  to prove $I(U(1))$ is Hausdorff. However $I$ is a subobject of 
the sum, so it suffices to show $f_K(\phi^KX)(U(1))$ is Hausdorff for all $K$.
In other words, we need to show that if $N$ is a Hausdorff $\cOcFK$-module
then $\cEi_K \cOcF \tensor_{\cOcFK}N$ is a Hausdorff $\cOcF$-module. 

For this, we note $\cOcF \tensor_{\cOcFK}N$ is Hausdorff since
the $F$-th summand is $H^*(BG/F) \tensor_{H^*(BG/KF)} N_{KF}$.  Since
$\cEi_K \cOcF \tensor_{\cOcFK}N$ is a colimit of these Hausdorff modules
under split monomorphisms by \ref{ffHausdorff}, it follows that it is Hausdorff.
\end{proof}

\begin{cor}
\label{faithfulflatness}
The inflation map $\cOcFK \lra \cOcF$ and the localization 
$\cOcFK \lra \cEi_K \cOcF$ are faithfully flat on the category of 
modules ocurring as $\phi^K M$. \qqed
\end{cor}

\part{The Adams spectral sequence.}

In Part 2 we  introduced the algebraic category $\cA (G)$, and 
in Part 3 we provide the connection with $G$-equivariant cohomology 
theories by defining the functor $\piA_*$ and constructing an Adams
spectral sequence based on it.

The basis of the connection is the calculation  
$$\cOcF =[\efp , \efp]^G_*$$
of the endomorphism ring of $\efp$. This is completed in 
Section \ref{sec:Endefp}. In preparation we begin by 
understanding the basic building blocks and how they are
related to each other. 

\section{Basic cells.}
\label{sec:Basic}

The familiar generators in topology are the natural cells $G/K_+$, 
but when working stably and rationally these break up further
using idempotents from the rationalized Burnside ring,  
and it is more convenient to consider the resulting pieces.
In more detail, smashing with $G/K_+$ gives a ring homomorphism 
$[S^0,S^0]^K \lra [G/K_+, G/K_+]^G$, and $[S^0,S^0]^K$ is isomorphic
 to the rationalized Burnside ring of finite $K$-sets $A(K)$. Indeed
the map 
$$\phi : [S^0,S^0]^K \stackrel{\cong} \lra \prod_{H} \Q$$ 
with $H$th component being given by the degree of geometric $H$-fixed points
is a ring isomorphism.

\begin{defn}
The {\em basic cell} for the closed subgroup $K$ is defined by 
$$\sK =e_{K/\Kz} G/K_+, $$
 where $e_{K/\Kz} \in A(K/\Kz)$ is the primitive idempotent
in the Burnside ring corresponding to the group $K/\Kz$ of components of $K$. 
\end{defn}

The usefulness of the basic cells is that they provide decompositions
of all the natural cells. 

\begin{lemma}
Suppose $\tK$ is a subgroup with identity component $\Kz$. 
There is a decomposition 
$$G/\tK_+ \simeq \bigvee_{\Kz \subseteq K \subseteq \tK} \sK$$
where the splitting is indexed by subgroups $K/\Kz$ of the group 
$\tK/\Kz$ of components of $\tK$.
\end{lemma}

\ppf\ We follow the pattern of \cite[2.1.5]{s1q}. It suffices to 
show that if $\Kz \subseteq K \subseteq \tK$ then 
$G_+ \sm_Ke_KS^0=\sK \simeq  G_+ \sm_{\tK} e_KS^0$. Indeed, we need only 
show that $G_+ \sm_{\tK} e_K \widetilde{\tK/K}$ is contractible
where $\widetilde{\tK/K}=\cofibre (\tK/K_+ \lra S^0)$.

We suppose $G=G' \times G''$ with $\tK \subseteq G'$ a product of inclusions
of a cyclic group in a circle. It suffices to prove the analogous result
with $G$ replaced by $G'$.
The analogue
of \cite[2.1.4]{s1q} replaces a single cofibre sequence by 
$r'=\rank (G')$ of them.
For each of the cyclic factors we apply the method of \cite[2.1.5]{s1q}
to the permutation representation of $\tK/K$.
\qqed

\begin{lemma}
Maps between basic cells in degree 0 are as follows
$$
[\sK , \sL]^G_0=
\left\{
\begin{array}{cc}
\Q & \mbox{ if $K$ is cotoral in $L$}\\
0  & \mbox{ otherwise} \end{array} \right. $$
\end{lemma}

\ppf\ We need only apply idempotents to the corresponding statements
with natural cells. Indeed $[G/K_+, \sL]^G=[S^0, \sL]^K$. This is zero unless
$L$ is of finite index in $K$. If $L$ is of finite index the idempotents
$e_K$ and $e_L$ are orthogonal unless $K=L$. Finally, if $K=L$ we note
that $G/L_+$ is $L$-equivariantly obtained from $S^0$ by attaching cells
of dimension $\geq 1$, and hence the desired group is a quotient of 
$e_K[S^0, S^0]^K=\Q$. It is non-trivial since $\sK$ is not contractible.
\qqed

\begin{lemma}
\label{endsF}
If $F$ is finite, the endomorphism ring of $\sF$ is exterior on $r$ generators,
$$[\sF, \sF]^G_*=\Lambda (H_1(G/F)) .$$
\end{lemma}

\ppf\ Additively the calculation is correct since
$[\sF, \sF]^G_*=[G/F_+, \sF]^G_*= 
[S^0, G/F_+]^F_*$ and $G/F_+$ is $F$-fixed and a torus with added basepoint. 
The ring structure may be seen by passing to non-equivariant homology
$$[\sF , \sF ]^G_* \lra \Hom (H_*(\sF ), H_*(\sF )). $$
This is a ring map and,  since  $H_*(\sF ) \cong H_*(G/F_+)$, 
the codomain is exterior. It remains to note that this is surjective in 
degree 1. This in turn follows from the rank 1 case by the K\"unneth theorem.
The rank 1 case is clear since the degree 1 map is tautologously detected
in $F$-equivariant homotopy and hence in homology.
\qqed

Finally we record the Whitehead Theorem for spectra with stable isotropy 
only at $F$.

\begin{lemma}
\label{FWhitehead}
If $F$ is finite, and  $X$ is a spectrum so that
\begin{enumerate}
\item $X$ has stable isotropy only at $F$ and
\item $[\sF , X ]^G_*=0 $ 
\end{enumerate}
then $X$ is contractible.
\end{lemma}

\ppf\ By (1) we have $\PK X $ trivial unless $K$ is a subgroup of $F$.
It therefore suffices to show that $[G/K_+, X]^G_*=0$ if $K \subseteq F$.
However, $G/K_+$ splits as a wedge of basic cells for finite subgroups. 
Since $X$ only has isotropy at $F$, the only possible contribution is 
from the summand $[\sF , X]^G_*$, and this is zero by hypothesis. 
\qqed

\section{Endomorphisms of injective spectra.}
\label{sec:Endefp}

The basis for the correspondence between algebra and topology is 
the universal space $\efp$ for the collection $\cF$ of finite
subgroups of $G$. This plays a central role because its 
endomorphism ring is so well behaved: the simplicity we see
here will have even more power in \cite{tnq3}.

We follow the strategy of \cite{s1q}, adapted to account for 
 the fact that the exterior algebra $H_*(G)$ and
the polynomial algebra $H^*(BG)$ now have $r$ generators, rather
than the single generator in the case of the circle. 

First it is convenient to introduce injective counterparts of the
basic cells.

\begin{defn}
\label{defn:erK}
For any subgroup $K$ we define the $G$-space $\erK$ by 
$$\erK = \cofibre (E[ \subset K]_+ \lra E[\subseteq K]_+).$$
\end{defn}

\begin{example}
(i) If $K=1$ we have $E \langle 1 \rangle =EG_+$.\\ 
(ii) If $K=G$ we have $E \langle G \rangle =\etp$ where $\cP$ is 
the family of proper subgroups of $G$.
\end{example}

Between them these give the general picture. 

\begin{lemma}
\label{lemma:erK}
If $K$ is a subgroup with identity component $\Kz$,
then there is an equivalence
$$\Phi^{\Kz}\erK \simeq E\langle K/\Kz \rangle$$
of $G/\Kz$-spaces and an equivalence 
$$\erK \simeq  S^{\infty V(\Kz)} \sm E\langle K/\Kz \rangle $$
of $G$-spaces. 
\end{lemma}

\begin{proof}
For any family $\mcH$ of subgroups and any subgroup $L$
$$\Phi^LE\mcH_+\simeq E\mcH/L_+$$
where $\mcH /L$ is the family of subgroups of $G/L$
which are images of those of $\mcH$. This gives the 
first statement. 

The second statement follows: the only group with non-trivial geometric
 fixed points on either side is $\Kz$, and for $\Kz$ the equivalence
 is the first statement.
\end{proof}

We are now ready to  identify homotopy endomorphism 
rings. 

\begin{thm}
\label{endef}
The homotopy endomorphism ring of $\ef$ is given by 
$$[\ef , \ef ]^G_*=\cOcF = \prod_{F \in \cF} H^*(BG/F) .$$
\end{thm}

\begin{proof}
First, as in \cite{EFQ} we may use idempotents to split $\efp$: 
$$\efp \simeq \bigvee_{F \in \cF} \erf .$$
It therefore suffices to prove the corresponding result, 
Theorem \ref{enderf},  about the summands.
\end{proof}

\begin{thm}
\label{enderf}
The homotopy endomorphism ring of $\erf$ is given by 
$$[\erf , \erf ]^G_* =  H^*(BG/F) .$$
\end{thm}

The first tool is a characterization of $\erf$.

\begin{prop}
\label{charerf}
If $F$ is finite, the spaces $\erf$ are characterized by 
\begin{enumerate}
\item  $\erf$ has stable isotropy only at $F$ and 
\item $[\sF , \erf ]^G_*=\Q$.
\end{enumerate}
\end{prop}

\ppf\ First note that $[\sF , \erf ]^G_*=\Q$. Now we proceed by cellular
approximation to construct a map $X \lra \erf$, where $X$ is constructed
from cells $\sF$ which is an isomorphism of $[\sF , \cdot ]^G_*$. 
This is an equivalence by the Whitehead Theorem \ref{FWhitehead}.
\qqed

We may now identify the endomorphism ring of $\erf$.

\begin{proof}[of \ref{enderf}]  
Note that $[\erf , \erf ]^G_*=[\erf , S^0]^G_*$ so the result will 
follow additively if we can construct $\erf$ with basic cells in even degrees
corresponding to the monomials in $H^*(BG/F)$. 
The proof is by killing homotopy groups.

In the proof of \ref{charerf} we noted 
 that $\erf$ can be constructed using the basic cell $\sF$. We repeat the 
proof, but this time keep track of the cells. 
By \ref{endsF}, the endomorphism ring of $\sF$ is exterior on $r$ generators.
Indeed, let 
$$ \Q \lla P_0 \lla P_1 \lla P_2 \lla \cdots $$
be the standard Koszul resolution of $\Q $ by free $\Lambda H_1(G/F)$-modules. 
Thus 
$$P=\Lambda (H_1G/F) [c_1,c_2, \ldots , c_r].$$
Note that the kernel 
of each map $P_n \lra P_{n-1}$ is generated by its bottom degree elements
and these are in bijective correspondence with monomials of degree $n$.

We argue inductively that we may construct (1) 
a $2n$-dimensional complex $X^{(2n)}$ with basic 
cells in bijective correspondence with monomials of degree $\leq n$ in 
the $c_1, c_2, \ldots , c_r$ so that its cellular chain complex is 
the first $n$ stages of the Koszul resolution  and (2)
a map $X^{(2n)} \lra \erf$ which is $2n$-connected. 
This is certainly true for $r=0$, so we need only describe the inductive step.
However, by construction the bottom degree homotopy generates $n$th syzygy in 
the Koszul resolution, so there is no obstruction.

It remains to comment on the ring structure. Consider the cellular filtration, 
and the resulting spectral sequence for $[S^0 , \cdot ]^G_*$. We obtain a ring
map 
$$[\erf , \erf]^G_* \lra \Hom (H^*(BG/F), H^*(BG/F)).$$
Each generator $c_i \in H^2(BG/F)$ corresponds to a map of resolutions, 
and we may realize this by a map $\erf \lra \Sigma^2 \erf$. 
It follows that the ring map is surjective. 
By the additive result,  it is an isomorphism. 
\end{proof}

We also need to know this identification is natural for quotient maps.

Now there is a natural map $\ef \lra \efk$ of $G$-spaces, since 
every finite subgroup of $G$ has finite image in $G/K$. Viewing this 
is a map of spectra and dualizing, we obtain 
a map $\decfk \lra \decf$. Combined with the inflation map 
$[S^0, \decfk ]^{G/K}_* \lra [S^0, \decfk ]^{G}_*$, and using
\ref{endef} for $G$ and $G/K$, 
 we obtain  a ring homomorphism $\cOcFK \lra \cOcF$.

\begin{lemma}
 The geometrically induced ring homomorphism coincides with the map 
 $q^*: \cOcFK \lra \cOcF$ described in Subsection \ref{subsec:FundInfl}, 
which is the product of the ring homomorphisms
$$q^*_{\tK}:\cO_{\tK} \lra \cO_{q^{-1}_*(\tK)}=\prod_{q_*(F)=\tK /K} 
H^*(BG/F)$$
where $q_*: \cF \lra \cF /K$ is reduction mod $K$, and the components of
$q^*_{\tK}$ are induced by the quotient maps
$G/F \lra G/\tK$.
\end{lemma}

\ppf\
We have the splitting $\ef \simeq \bigvee_{F\in \cF}  \erf$ of rational 
$G$-spectra \cite{EFQ}. Similarly the stable rational $G/K$-splitting
$\efk \simeq \bigvee_{\tK /K \in \cF /K}  \etKK$ may be inflated to a 
$G$-splitting. 
From fixed points one sees that the map $\ef \lra \efk$ respects the 
splitting in the sense that $\erf$ maps trivially to $\etKK$ unless 
$q(F)=\tK /K$. Since duality takes sums to products, 
\ref{enderf} completes the proof. 
\qqed

\section{Topology of Euler classes.}
\label{sec:Euler}

The next ingredient is to show that the inclusions $S^0\lra S^V$ 
induce suitable Euler classes. 

The  relevant input from topology comes from 
 the  Thom isomorphism for an individual stalk.
We once again use the basic injectives \ref{defn:erK}.

\begin{lemma}
For any finite group $F$ there is an equivalence
 $$S^V \sm \erf \simeq S^{| V^F|} \sm \erf .$$
\end{lemma}

\ppf\ The cofibre of the map $S^{V^F} \lra S^V$ 
is built from cells with isotropy not containing $F$. 
It is therefore  contractible when smashed with $\erf$. 
We may thus suppose $V$ is $F$-fixed. 

Now $\erf$ may be built from basic cells $\sF$. Since
$$G/F_+ \sm S^{V^F} \simeq G_+ \sm_F S^{|V^F|}
\simeq G/F_+ \sm S^{|V^F|}, $$
we find that  
$$\sF \sm S^{V^F} \simeq  \sF \sm S^{|V^F|}.$$ 
Accordingly,  $\erf \sm S^{V^F}$ is also built from 
cells $\sF$ and 
$$[\sF , \erf \sm S^{V^F}]^G_*=
[\sF , \erf \sm S^{|V^F|}]^G_*=\Sigma^{|V^F|}\Q.$$
Thus the result follows from \ref{charerf}.
\qqed

\begin{remark} 
Note that the proof displays a specific equivalence
on the bottom cell, and hence determines the 
homotopy class of the equivalence.
\end{remark}

As usual, the Thom isomorphism gives rise to an Euler class.

\begin{defn}
The $F$ Euler class $c(V)(F)$ of a representation $V$ is the map 
$$S^0 \sm \erf \lra S^V \sm \erf \simeq S^{|V^F|} \sm \erf. $$
\end{defn}

We may identify these Euler classes in familiar terms. 

\begin{lemma}
Under the identification $[\erf , \erf]^G_* =H^*(BG/F)$, 
the Euler class $c(V)(F)$ is the ordinary cohomology 
Euler class $c_H(V^F)$.
\end{lemma}

\begin{proof} 
Since both Euler classes take sums of representations to products, 
it  suffices to consider a 1-dimensional representation $V$. 
If $V^F=0$, both Euler classes are 1. If $V$ is fixed by $F$,
then $V$ is a faithful representation of $G/K$ for some $(r-1)$-dimensional
subgroup $K$ containing $F$. Both maps are given by multiplication by 
a degree 2 class.

It therefore suffices to consider the  case of the circle and the 
representation $z^n$. The standard generator is the first Euler 
class $c_H(z)$
and the additive formal group shows $c_H(z^n)=nc_H(z)$. 
On the other hand, the identification of 
$G_+ \sm S^{z^n} \simeq G_+ \sm S^2$
lets $t(g\sm x)=tg \sm x$ in $G_+\sm S^2$
correspond to   $t(g\sm x)=tg \sm t^nx$ in $G_+\sm S^{z^n}$, which 
is a map of degree $n$. 
\end{proof}

In view of the splitting theorem $\ef \simeq \bigvee_{F \in \cF}\erf$ 
we obtain a general Thom isomorphism.

\begin{cor}
For any virtual complex representation $V$ and associated dimension function 
$v: \cF \lra \Z$  defined by $v(F)= \dim_{\bR}(V^F)$, there
are equivalences 
$$S^V \sm \efp \simeq \bigvee_F S^{v(F)}\sm \erf, $$
and
$$S^V \sm \defp \simeq \prod_F S^{v(F)}\sm D\erf.$$
\end{cor} 

We may now define the global Euler class.

\begin{defn}
The Euler class of a complex representation $V$ is 
$$S^0 \sm \efp \lra S^V \sm \efp \simeq \bigvee_F S^{v(F)}\sm \erf,$$
as a non-homogeneous element of $\cOcF$.
\end{defn}

\begin{cor}
The Euler class, viewed as an element of $\cOcF$
has $F$th component.
$$c(V)(F)=c_H(V^F) \in H^*(BG/F).\qqed$$
\end{cor}

\section{Sheaves from spectra.}
\label{sec:piA}

Now that we understand the homotopy endomorphism ring
of $\efp$ we may forge the link with algebra:
since $[\efp , \efp ]^G_*=\cOcF$ by \ref{endef}, 
any spectrum $X \sm \decf$ has homotopy groups which are
$\cOcF$-modules.  In this section we give the proof that
$\piA_*$ takes values in $\cA (G)$ (stated as \ref{piAinAG}).

From  the definition of Euler classes we see that
$\piA_*(X)$ is  quasi-coherent.

\begin{prop}
\label{piAqc}
For any $G$-spectrum $X$ the object $\piA_*(X)$ is 
quasi-coherent in the sense that for any connected subgroup $K$, 
there is an isomorphism
$$\piA_*(X)(U(K))\cong \cEi_K \piA_*(X)(U(1)).$$
natural in $X$.
\end{prop}

\begin{proof}
We combine the definition of $\piA_*$ with that of Euler classes to 
obtain
$$\piA_*(X)(U(K))=\piG_*(X \sm \defp \sm \siftyK)
=\cEi_K\piG_*(X \sm \defp)=\cEi_K \piA_*(X)(U(1)).$$
\end{proof}

We may now complete the proof that $S^0$ corresponds to the structure
sheaf $\cO$.

\begin{proof}[of Theorem \ref{piASisO}]
Recall that 
$$\piA_*(X)(U(K))=\pi^G_*(D\efp \sm \siftyK \sm X). $$
Taking $K=1$, we see that by \ref{endef}, $\piA_*(S^0)(U(1))=\cOcF$. 
By \ref{piAqc} $$\piA_*(S^0)(U(K))=\cEi_K \cOcF =\cO (U(K)).$$
\end{proof}

It is now evident from \ref{piAqc} that $\piA_*(X)$ is a sheaf of 
$\cO$-modules, and quasi-coherent in that sense. To complete the proof
that $\piA_*$ takes values in $\cA (G)$, it remains to 
show that $\piA_*(X)$ is extended.

\begin{lemma}
\label{picAatH}
The quasi-coherent $U$-sheaf $\piA_*(X)$ of $\cO$-modules is extended.
In fact, the value at $U(K)$ splits with 
$$\phi^K\piA_*(X)=\piGK_*(\PK X \sm \decfk )$$
since there is an isomorphism 
$$\piG_*(X \sm \defp \sm \siftyK )\cong
\cEi_K \cO_{\cF }\otimes_{\cOcFK}
\piGK_*(\PK X \sm \decfk )$$
natural in $X$.
\end{lemma}

\ppf\ There is a natural transformation arising from 
$$\inflGKG (\PK X \sm \decfk ) \lra X \sm \defp \sm \siftyK . $$
This gives a natural transformation of homology theories of $X$, so 
 we need only  check it is an isomorphism for various cells
$X=G/H_+$. If $X=S^0=G/G_+$ the map is an isomorphism by definition. 
The general case follows by the $Rep(G)$-isomorphism 
argument (Theorem \ref{RepG})  since we have Thom  isomorphisms on both sides.
\qqed

\section{Adams spectral sequences.}
\label{secASS}

It is clear that $\piA_*$ is  functorial and exact, and  therefore 
by \ref{piAinAG} it defines  a homology functor
$$\piA_*: \Gspec \lra \cA (G) .$$
with values in the abelian category $\cA (G)$ of finite injective dimension.

\begin{thm}
The homology theory $\piA_*$  gives a convergent Adams spectral sequence
which collapses by  $E_{2r+1}$.
\end{thm} 

We apply the usual method for constructing an Adams spectral sequence
based on a homology theory $H_*$ on a category $\C$ with values in an 
abelian category $\cA$ (in our case $\C$ is the category of rational
$G$-spectra, $\cA =\cA (G)$ and $H_*=\piA_*$). It may be helpful to 
summarize the process: to construct an Adams spectral 
sequence for calculating $[X,Y]$ we proceed as follows. 

\noindent
{\bf Step 0:} Take an injective resolution 
$$0 \lra H_*(Y) \lra I_0 \lra I_1 \lra I_2 \lra \cdots $$
in $\cA$.  

\noindent
{\bf Step 1:} Show that enough injectives $I$ of $\cA$ (including the $I_j$)
can be realized by  objects $\bI$ of $\C$ in the sense that $H_*(\bI)\cong I$.
 
\noindent
{\bf Step 2:} Show that the injective case of the spectral sequence is 
correct in that homology gives an isomorphism
$$[X,Y] \stackrel{\cong}\lra \Hom_{\cA}(H_*(X),H_*(Y))$$
if $Y$ is one of the injectives $\bI$ used in Step 1.

\noindent
{\bf Step 3:} Now construct the Adams tower 
$$\diagram 
\vdots &\\
.\dto &\\
Y_2 \rto \dto & \Sigma^{-2} \bI_2\\
Y_1 \rto \dto & \Sigma^{-1} \bI_1\\
Y_0 \rto      & \Sigma^0 \bI_0\\
\enddiagram$$
over $Y=Y_0$ from the resolution. This is a formality from Step 2. 
We work up the tower, at each stage defining $Y_{j+1}$ to be
the fibre of $Y_j \lra \Sigma^{-j}\bI_j$, and noting that $H_*(Y_{j+1})$ 
is the $(j+1)$st syzygy of $H_*(Y)$.

\noindent
{\bf Step 4:} Apply $[X, \cdot ]$ to the tower. By the injective case
(Step 2), we identify the $E_1$ term with the complex $\Hom_{\cA}^*(H_*(X),I_{\bullet})$
and the $E_2$ term with $\Ext_{\cA}^{*,*}(H_*(X),H_*(Y))$.

\noindent
{\bf Step 5a:}  If the injective resolution is infinite, the first
step of convergence is to show that $H_*(\holim_j Y_j)$ is calculated
using a Milnor exact sequence from the inverse system $\{ H_*(Y_j)\}_j$,
and hence that $H_*(\holim_j Y_j)=0$. In our case, this is automatic, 
since the resolution is finite.

\noindent
{\bf Step 5b:} Deduce convergence from Step 5a by showing $\holim_j Y_j\simeq *$.
In other words we must show that $H_*(\cdot )$ detects isomorphisms
in the sense that $H_*(Z)=0$ implies $Z \simeq *$. In general, 
one needs to require that $\C$ is a category of appropriately complete objects for 
this to be true. This establishes conditional convergence. If $\cA$ has finite
injective dimension, finite convergence is then immediate.

\vspace*{2ex}
In our case Step 0 follows from Theorem \ref{injdimisrank}, and 
 Lemma \ref{enoughinjectives} shows only a small list of injectives
is required. Steps 3 and 4 are formalities, 
and Step 5a is automatic since $\cA$ has finite injective dimension
(\ref{injdimisrank}). This leaves us to complete Steps 1, 2 and  5b.

For Step 1,   we use the basic injective $G$-spectrum 
$\erK$ of \ref{defn:erK} which has stable isotropy only at the subgroup $K$.
The essential property is that this realizes the basic injectives $I(K)$
in $\cA (G)$. 

\begin{lemma}
\label{piAehisinj}
If $K$ is of codimension $c$ we have
$$I(K)=f_{\Kz}(\Sigma^cH_*(BG/K))=\piA_*(\erf )$$
and hence by \ref{enoughinjectives} there are enough realizable injectives.
\end{lemma}

\ppf\ First, by \ref{lemma:erK},  we have
$$\erK =S^{\infty V(\Kz)} \sm E \langle K/\Kz \rangle,  $$
so that
$$\pi^G_*(D\efp \sm \erK)=\cEi_{\Kz} \cOcF \otimes_{\cOcFK} \pi^{G/\Kz}_*
(DE(\cF /\Kz)_+ \sm E \langle K/\Kz \rangle).$$
The result therefore follows from the special case in which $K=F$ is
finite and $c=r$.

Now, since $S^0 \lra D\efp$ is an $F$-equivalence, 
$D\efp \sm \erf \simeq \erf$ and therefore
$$\pi^G_*(D\efp \sm \erf)=\Sigma^r H_*(BG/F).$$
Since this is a torsion module
$$\piA_*(\efp)=f_1(\Sigma^r H_*(BG/F)).$$ 
\qqed

For Step 2 we prove the injective case of the Adams spectral sequence. 

\begin{lemma}
For any $G$-spectrum $X$, application of $\piA_*$ induces an 
isomorphism
$$[X,\erK ]^G \stackrel{\cong}\lra \Hom_{\cA}(\piA_*(X),\piA_*(\erK)). $$
\end{lemma}

\ppf\ Let $N=\piA_*(X)$, and  argue by induction on the dimension of $G$.

For $\erK$ we combine the following diagram
$$\diagram
[X, \erK ]^G \dto^{\cong} \rto & \Hom_{\cA} (N, f_K(\Sigma^c H_*(BG/K ))\dto_{\cong} \\
[\PK X, EG/K_+ ]^{G/K} \rto &\Hom_{H^*(BG/K)} (\phi^KN, \Sigma^c H_*(BG/K))
\enddiagram$$
with a result for $G/K$ to show the bottom
horizontal is an isomorphism. 

For notational simplicity we treat the case $K=1$, where
we are left to show
$$\piG_*: 
[X,\eg ]^G \stackrel{\cong}\lra \Hom_{H^*(BG)} 
(\pi^G_*(X \sm \deg ), \Sigma^r H_*(BG) ).$$
It is easy to see the groups are isomorphic for $X=S^0$. Passage to homology
is injective because $\piG_*(S^0) \lra \piG_*(S^0 \sm \deg )$ is a 
monomorphism in degree 0. Since $\piG_*$
 compares rational vector spaces of equal finite
dimension when $X=S^0$, it is an isomorphism. 
There are Thom isomorphisms in algebra
and topology, so it follows that passage to homology is an isomorphism 
for $X=S^V$ for any complex representation $V$. 
By the $Rep(G)$-isomorphism argument (Theorem \ref{RepG}) it is an isomorphism for
$X=G/K_+$ for any subgroup $K$ and hence in general. \qqed

Finally, for Step 5b we prove the universal Whitehead Theorem.

\begin{lemma}
\label{univWhitehead}
The functor $\piA_*$ detects isomorphisms in the sense that
if $f: Y \lra Z$ is a map of $G$ spectra inducing an isomorphism
$f_* : \piA_*(Y ) \lra \piA_*(Z)$ then $f$ is an equivalence. 
\end{lemma}

\ppf\  Since $\piA_*$ is exact, it suffices to prove that if $\piA_*(X)=0 $
then $X \simeq *$. We argue by induction on the dimension of $G$.

Suppose $\piA_*(X)=0$. From the geometric fixed point Whitehead
theorem (extend the treatment of \cite{assiet} to infinite compact
Lie groups or deduce the result from the Lewis-May fixed point
Whitehead Theorem \cite{lmsm}) it suffices to show that $\PK X$ 
is  non-equivariantly contractible for all $K$. 
Since
$\phi^K\piA_*(X)= \piGK_*(\PK X)$, and $\cEi_K \cOcF $ is faithfully 
flat over $\cOcFK$ by \ref{faithfulflatness},  the result follows by 
induction from the $G/K$-equivariant result provided $K \neq 1$.

It remains to deduce $X \sm \eg$ is contractible. Since we have 
Thom isomorphisms we note that by the $Rep(G)$-isomorphism argument 
(Theorem \ref{RepG}), it suffices to show that $\piG_*(X \sm \eg)=0$. 
Now $\ef \lra S^0$ is a non-equivariant equivalence so 
that $\decf \sm \eg \simeq DS^0 \sm \eg =\eg$, so that it
is enough to show $\piG_*(X \sm \decf \sm \eg )=0$.

However if $Y$ has Thom isomorphisms and $\bC$ is a one dimensional 
representation then the cofibre sequence 
$$Y \sm \freec \lra Y \lra Y \sm \siftyc$$
shows that if $\piG_*(Y)=0$ then also $\piG_*(Y \sm \freec )=0$.
Writing $\eg =S(\infty \bC_1)_+\sm \ldots \sm S(\infty \bC_r)_+$
we reach the desired conclusion in $r$ steps.
\qqed

\section{The $Rep(G)$-isomorphism argument.}
\label{sec:RepG}

The present section records a method that is useful rather
generally in equivariant topology. It has nothing to do with the fact
that we are working rationally. 

When trying to establish an object is contractible or a map is 
an equivalence we want to use the most convenient test objects. 
The Whitehead Theorem 
says it suffices to use the set $\{ G/H_+ \; | \; H \subseteq G\}$:
if $[G/H_+,X]^G_*=\piH_*(X)=0$ for all $H$ then $X$ is contractible. 

\begin{defn}
A $G$-spectrum is $Rep (G)$-contractible if $[S^V,X]_*^G=0$ 
for all complex representations $V$.
\end{defn}

It is often easy to see $\piG_*(X)=0$. If we happen to have Thom 
isomorphisms $S^V \sm X \simeq S^{|V|} \sm X$ for complex representations
$V$ this shows that $X$ is $Rep(G)$-contractible. 
It does not necessarily follow that
$X$ is contractible  even if $X$ is rational and $G$
is abelian (for example if $G$ is cyclic of order $3$ and $X$ is a Moore
spectrum for a two dimensional simple representation \cite{assiet}) 
but it is useful to have a sufficient condition.

\begin{thm}
\label{RepG}
Suppose $G$ is an abelian compact Lie group.
If $X$ is a $Rep(G)$ contractible $G$-spectrum
and $G/H$ acts trivially on $\piH_*(S^{-V} \sm X)$ for all $H$ and all
desuspensions of $X$ then 
$X$ is contractible.  If $G$ is a torus,  the condition of trivial 
action may be omitted.
\end{thm}

\ppf\ We argue by induction on the size of $G$: since compact Lie
groups satisfy the descending chain condition on subgroups we can 
assume the result is true for all proper subgroups. 
We know $\piG_*(X)=0$ by hypothesis, so by the Whitehead Theorem 
it suffices to show that $\piK_*(X)=0$ for all proper subgroups.
Now any proper subgroup $K$ lies in a subgroup $H$ with 
$G/H$ a subgroup of the circle. It therefore suffices by induction 
to establish that $X$ is $Rep(H)$-contractible. Since $H \subseteq G$, 
any trivial action condition will certainly be inherited by subgroups. 

If $G/H$ is a circle,  we use the cofibre sequence
$G/H_+ \lra S^0 \lra S^{V(H)}$ where $V(H)$ has kernel $H$.
We conclude that $[G/H_+,X]^G_*=0$, and more generally, by smashing with $S^V$,  
that $[S^V,X]^H_*=0$ for any representation $V$ of $G$.
Since $G$ is abelian, every representation of $H$ extends to one
of $G$, and so $X$ is $Rep(H)$-contractible, and hence we conclude $X$
is $H$-contractible by induction.

If $G/H$ is a finite cyclic group we choose a faithful representation 
$W(H)$ of $G/H$ and use the cofibre sequence 
$S(W(H))_+ \lra S^0 \lra S^{W(H)}$ and the stable cofibre sequence
$$G/H_+ \stackrel{1-g}\lra G/H_+ \lra S(W(H))_+, $$
where $g$ is a generator of $G/H_+$.
 The first shows that $[S(W(H))_+,X]^G_*=0$, 
and the second shows that $1-g$ gives an isomorphism of 
$[G/H_+, X]^G_*$. By the trivial action condition we conclude
$[G/H_+,X]^G_*=0$, and more generally $[S^V,X]^H_*=0$.
\qqed

\begin{remark}
(i) It suffices to assume that $G$ acts unipotently on $\piH_*(X)$
for all $H$. This is useful for $p$-groups in characteristic $p$.\\
(ii) Variants on this theorem are useful in other contexts. For 
instance any nilpotent or supersoluble finite group has maximal 
subgroups which are normal with cyclic quotient. However, not every 
representation of  a maximal subgroup extends to one for $G$, so 
additional hypotheses are necessary.\\
(iii) If we admit real representations, then no trivial action condition 
is necessary for subgroups of index 2 since the mapping cone of 
$G/H_+ \lra S^0$ is $S^V$ for a real representation $V$.
\end{remark}

\newcommand{\piAs}{\piA_*(\sH)}
\newcommand{\algconnH}{\mathrm{algconn}_H}
\newcommand{\ExtA}{\mathrm{Ext}_{\cA}}
\newcommand{\ExtAsts}[1]{\mathrm{Ext}_{\cA}^{s,t}(\piAs, #1)}
\newcommand{\fmax}{\mathrm{fmax}}

\section{A characterization of basic cells.}

We show that basic cells are characterized by their
homotopy. Apart from illustrating the use of the Adams
spectral sequence, this result is needed in \cite{tnq3}.

\begin{thm}
\label{charcells}
If $\piA_*(X) \cong \piA_*(\sH)$ then $X \simeq \sH$.
\end{thm}

\begin{proof}
We apply the Adams spectral sequence  to calculate
$[\sH, X]^G_0$. In the rest of the section we will show
$$\ExtA^{s,t}(\piAs, \piAs )=0 \mbox{ for } t-s<s$$ 
(and hence in particular for $t-s<0$) so that the isomorphism  
$\piAs \stackrel{\cong}\lra \piA_*(X)$
lifts to a map $\sH \lra X$. Since $\piA_*$ detects equivalences, 
it follows that $\sH \simeq X$ as claimed. 
\end{proof}


\subsection{Resolutions of universal spaces.}
\label{subsec:resolutions}

Although the vanishing result we require is purely algebraic, it 
is illuminating to give a geometric realization. In later subsections
we will only use the algebraic exact sequences that arise by 
applying $\piA_*$ to the sequences of spectra.

\begin{prop}
\label{Koszulresn}
If $\dim (G/H)=d$ there is a sequence
$$G/H_+ \lra EG/H_+ \lra \binom{d}{1} \Sigma^2 EG/H_+ \lra 
\binom{d}{2} \Sigma^4 EG/H_+ \lra \cdots
\lra \Sigma^{2d} EG/H_+ $$
inducing an exact sequence in $\piA_*$.
\end{prop}

\begin{proof}
Note the individual entries  $\piA_*(X)(K)$ are obtained by first
passing to geometric $K$-fixed points, and then extending scalars.
Next observe that for both $G/H_+$ and $EG/H_+$ these fixed points
are either both contractible or both copies of the corresponding spaces
for the quotient group $G/K$. It therefore suffices to deal with 
the special case $H=1$ and  show that it is exact in homotopy.

The proof  is based on algebra over 
the polynomial ring of $H^*(BG)$. Indeed, we may form an 
injective resolution of $\Q$ as the 
Matlis dual of the Koszul resolution for the polynomial ring of $H^*(BG)$.

First note that $[EG_+,EG_+]^G_*=H^*(BG)$ is a polynomial ring
on $r$ generators $x_1,\ldots, x_r$ of degree 2. The augmented
Koszul complex takes the form
$$\Q \lla F_0 \lla F_1 \lla F_2 \lla \cdots \lla F_r, $$
where
$$F_i=\binom{r}{i}\Sigma^{2i}H^*(BG), $$
and is exact. Dualizing, with respect to $\Q$ we obtain 
$$\Q \lra F^{\vee}_0 \lra F^{\vee}_1 \lra F^{\vee}_2 \lra \cdots \lra 
F^{\vee}_r .$$
Now using the identification of $[EG_+, EG_+]^G_*$ we realize this as
$$\sH \lra I_0 \lra I_1 \lra \cdots \lra I_r$$
where
$$I_i=\binom{r}{i}\Sigma^{2i}EG_+.$$
In this case (i.e., with $H=1$), the exactness of the dual Koszul 
complex immediately gives exactness in the category $\cA$.
\end{proof}


Now we need to deal with universal spaces $E\cF_+$, and 
will construct a resolution using a filtration of $\cF$.
To describe one step,  let  $\fmax (\cF)$ consist of the
subgroups of finite index in a maximal subgroup of $\cF$, 
and take $\cF^1 := \cF \setminus \fmax (\cF)$. Finally, 
we define the codimension  filtration of $\cF$ by 
$\cF^i=(\cF^{i-1})^1$.

\begin{prop}
\label{codimresn}
The Adams resolution of $E\cF_+$ is
$$E\cF_+ 
\lra \bigvee_{H \in \fmax (\cF)}E\lr{H}
\lra \bigvee_{H \in \fmax (\cF^1)}\Sigma^1E\lr{H}
\lra \cdots 
\lra \bigvee_{H \in \fmax (\cF^r)}\Sigma^rE\lr{H}$$
\end{prop}

\begin{proof}
It is enough to  deal with a single step since the general 
case follows by repeating it.

\begin{lemma}
There is a cofibre sequence
$$E\cF_+ \lra \bigvee_{H \in \fmax(\cF)} E\lr{H} \lra \Sigma E\cF^1_+, $$
which induces a short exact sequence in $\piA_*$. 
\end{lemma}
\begin{proof}
To see
the cofibre sequence exists we need only show that the 
cofibre $C$ of the universal map $E\cF^1_+\lra E\cF_+$ splits as
a wedge as indicated. Indeed, we may construct a map 
$C \lra E\lr{H}$ which is an equivalence on $H$-fixed points
for each $H \in \fmax (\cF)$ by obstruction theory, and hence a map 
$C \lra \bigvee_{H \in \fmax(\cF)} E\lr{H}, $
since the sum coincides with the product up to homotopy.
 It is an equivalence by the Whitehead Theorem. To see that the cofibre
sequence gives an algebraic short exact sequence we observe that 
$E\cF^1_+ \lra E\cF_+$ must be zero, since $\piA_*(E\cF^1_+)$ is
$\fmax (\cF)$-torsion in the sense that if  $H \in \cF^1$ lies
in $K \in \fmax (\cF)$ then 
$$\cEi_{H/K}\piA_*(E\cF^1_+)(H)=\piA_*(E\cF^1_+)(K)=0, $$
whereas $\cE_{H/K}$ consists of non-zero divisors in  
$\piA_*(E\cF_+)(H)$.
\end{proof}

\end{proof}

\subsection{Algebraic connectivity.}
It is convenient to have some terminology to discuss the algebraic
counterparts of natural homotopy theoretic constructions. 

\begin{defn}
(i) We say that a spectrum $X$ is {\em algebraically $c$-connected at $H$} if 
$$\ExtAsts{\piA_*(X)}=0 \mbox{ for } t-s\leq c.$$
 If $X$ is
$c$-connected  but not $(c+1)$-connected, 
 we write $\algconnH (X)=c$, and if it is $c$-connected
 for all integers $c$ we write $\algconnH (X)=\infty$. 

(ii) We say that $X$ has
{\em vanishing line of slope 1} at $H$ if 
 $\ExtAsts{\piA_*(X)}=0$ for $t-s\leq s+\algconnH (X)$. 
\end{defn}

\begin{remark}
(i) If $X$ has  algebraic $H$-connectivity $c$ then 
$\Sigma^nX$ has algebraic $H$-connectivity $c+n$. 

(ii) If $X$ has vanishing line of slope 1 then 
in particular the non-vanishing Ext group is on the 0-line.

(iii) By the Adams 
spectral sequence, the first non-zero $\sH$-homotopy group of $X$ is
in dimension $\algconnH(X)+1$, and if $X$ has vanishing line of 
slope 1 then 
$$[\sH , X]^G_{\algconnH(X)+1}=\HomA (\piAs, \piA_*(X)).$$
\end{remark}

We may manipulate algebraic connectivity with a slope 1 vanishing
line, rather like homotopy groups.

\begin{lemma}
\label{ses}
If the sequence of spectra $X \lra Y \lra Z$ induces
a short exact sequence
$$0 \lra \piA_*(X)\lra \piA_*(Y)\lra \piA_*(Z)\lra 0$$
where $Y$ and $Z$ have vanishing lines of slope 1 and
$\algconnH (Y)=c, \algconnH(Z)=c+2$ then $X$ has a vanishing
line of slope 1 and $\algconnH(X)=c$.\qqed
\end{lemma}

This allows us to give estimates using resolutions.

\begin{cor}
\label{les}
If the sequence of spectra
$$X \lra Y_0 \lra Y_1 \lra Y_2 \lra \cdots$$ induces
an exact sequence
$$0 \lra \piA_*(X)\lra \piA_*(Y_0)\lra \piA_*(Y_1)\lra \piA_*(Y_2)\lra 
\cdots $$
where $Y_i$ has a  vanishing line of slope 1 and
$ \algconnH(Y_i)=c+2i$ then $X$ has a vanishing
line of slope 1 and $\algconnH(X)=c$.\qqed
\end{cor}

\subsection{Estimates of connectivity}
Finally we may use the algebraic resolutions from Subsection 
\ref{subsec:resolutions} to estimate algebraic connectivity.
The starting point is the fact that $E\lr{K}$ is injective so 
it is easy to determine connectivity, for example by using 
the known homotopy groups of universal spaces. 

\begin{lemma}
The spectrum $E\lr{K}$ has a slope 1 vanishing line and 
$$\algconnH (E\lr{K})+1=
\dichotomy{\dim (H/K)  &\mbox{ if } K \subseteq H}
{\infty &\mbox{ if } K \not \subseteq H.\qqed}$$
\end{lemma}

From this we may prove the required result by connecting 
$\sH$ to objects $E\lr{K}$ by the cofibre sequences of 
Subsection \ref{subsec:resolutions}, and applying \ref{les}. 
One expects to be able to apply \ref{les} only when
the bottom homotopy comes from the displayed sequences, and
not from the connecting homomorphism.
 
\begin{lemma}
The spectrum $EG/K_+$ has a slope 1 vanishing line and 
$$\algconnH (EG/K_+) +1 =\dim (H/K)  \mbox{ if } K \subseteq H.$$
\end{lemma}

\begin{proof}
We use the resolution described in \ref{codimresn}.
The family in question consists of subgroups of $K$. Accordingly
the $i$th term in the codimension filtration consists of groups
of codimension $i$ in $K$. Thus the algebraic connectivity 
of a term $\Sigma^iE\lr{L}$ in the $i$th stage of the resolution 
is $\dim(H/L)+i=\dim(H/K)+2i$ if $L \subseteq H$ (and $\infty$ otherwise).
This provides the hypothesis we need  to apply \ref{les}.
\end{proof}

The Koszul-type resolution of \ref{Koszulresn} gives the required
estimate for $G/K_+$.

\begin{cor}
The spectrum $G/K_+$ has a slope 1 vanishing line and 
$$\algconnH (G/K_+)+1=\dim (H/K)  \mbox{ if } K \subseteq H.\qqed $$
\end{cor}

If $H$ is connected then $\sH=G/H_+$, and in general 
$\sH$ is a retract of $G/H_+$, so we obtain the estimate required
for Theorem \ref{charcells}.

\end{document}